\documentclass[12pt,a4paper]{article}
\usepackage[utf8]{inputenc}
\usepackage{tikz-network}
\usepackage{pgfplots}
\usepackage[english]{babel}
\usepackage{amsthm}
\usepackage{amsmath}
\usepackage{authblk}
\numberwithin{equation}{section}
\usepackage{amssymb}
\usepackage{enumitem}
\usepackage{mathtools}
\usepackage{graphics}
\usepackage{graphicx}
\usepackage{pifont}
\usepackage[export]{adjustbox}
\usepackage{geometry}
\usepackage{hyperref}
\newtheorem{thm}{Theorem}[section]
\newtheorem{prop}[thm]{Proposition}
\newtheorem{lem}[thm]{Lemma}
\newtheorem{cor}[thm]{Corollary}
\newtheorem{defi}[thm]{Definition}

\newtheorem{rem}[thm]{Remark}

\newcommand{\pp}{\mathbin{\ooalign{\raisebox{1.6ex}{$+$}\cr\raisebox{-0.1ex}{$+$}}}}
\newcommand{\mm}{\mathbin{\ooalign{\raisebox{1.2ex}{$-$}\cr\raisebox{0ex}{$-$}}}}
\newcommand{\pd}{\mathbin{\ooalign{\raisebox{1ex}{$+$}\cr\raisebox{-0.1ex}{$-$}}}}

\usetikzlibrary{shapes,arrows,positioning}
\title{$r$-orientation of a signed graph and its application on coronae of signed graphs}
\author[1]{Satyam Guragain\thanks{Email: shatym17@gmail.com}}
\author[2]{Ravi Srivastava\thanks{Corresponding author, Email: ravi@nitsikkim.ac.in}}
\author[3]{Bishal Sonar\thanks{Email: bsonarnits@gmail.com}}
\affil[1,2,3]{Department of Mathematics, National Institute of Technology Sikkim, South Sikkim 737139, India}
\date{}
\pgfplotsset{compat=1.18}
\begin{document}

\maketitle
\begin{abstract}
\noindent For unsigned graphs $G$ and $H$, the characteristic polynomial of different graph matrices for edge corona ($G\diamond H$), subdivision vertex neighbourhood corona ($G\boxdot H$) and subdivision edge neighbourhood corona ($G \boxminus H$) has already been studied using the concept of coronal. However, till date no work regarding the spectrum of these products has been studied for signed graphs. In our work, we have filled this gap and defined these variants  of coronae by introducing the concept of reverse orientation ($r$-orientation). We analyzed the structural properties of these product. Also, the characteristic polynomial of adjacency matrix, Laplacian matrices (signed and signless) and normalized Laplacian matrix of these variants of corona product of regular signed graphs under $r$-orientation is obtained using the concept of signed coronal. These results help us to construct infinitely many families of pairs of cospectral signed graphs.
\end{abstract}
\section{Introduction}
All graphs considered throughout this paper are undirected, simple and finite. Let $G=(V,E)$ be a graph with vertex set $V=V(G)=\{v_1,v_2,\cdots,v_n\}$ and edge set $E=E(G)=\{e_1,e_2,\cdots,e_m\}$. A signed graph $\Gamma=(G,\sigma)$ consists of an unsigned graph $G=(V,E)$ and a mapping function $\sigma: E(G)\rightarrow \{+1,-1\}$. This mapping, referred to as the signature of $\Gamma$, assigns either a positive or negative sign to each edge within the graph. The signed degree of a vertex $v$, denoted as $sdeg(v)$, is determined by subtracting the negative degree $d_v^-$ from the positive degree $d_v^+$. The total degree of a vertex $v$, denoted as $d_v$, is obtained by summing $d_v^+$ and $d_v^-$. For a signed graph $\Gamma=(G,\sigma)$, the adjacency matrix is represented by the $n \times n$ matrix denoted as $A(\Gamma)$, with elements given by $a_{ij}^\sigma=\sigma(v_i v_j) a_{ij}$, where $a_{ij}=1$ if vertices $v_i$ and $v_j$ are adjacent ($v_i\sim v_j$), and $0$ otherwise. The spectrum of $A(\Gamma)$ is also known as the spectrum of $\Gamma$.
The signed Laplacian matrix and the signless Laplacian matrix of $\Gamma$ is given by $L(\Gamma) = D(\Gamma) - A(\Gamma)$ and $Q(\Gamma) = D(\Gamma) + A(\Gamma)$ respectively where $D(\Gamma)$ is the diagonal matrix of the vertex degree of $\Gamma$. 
The normalized Laplacian of $\Gamma=(G,\sigma)$ is defined as $\mathbb{L}(\Gamma)=D(\Gamma)^{-\frac{1}{2}}L(\Gamma)D(\Gamma)^{-\frac{1}{2}}$ where,\\
\[
(D(\Gamma)^{-\frac{1}{2}})_{ij}=\begin{cases}
    0 &\text{ if } i\neq j\\
    \frac{1}{\sqrt{d_{u_{_i}}}} &\text{ if } i=j \text{ and } d_{u_{_i}}\neq 0\\
    0 &\text{ if } i=j \text{ and } d_{u_{_i}}= 0
\end{cases}
\]
If $\Gamma$ is $\gamma$-regular signed graph on $n$ vertices then $\mathbb{L}(\Gamma)=I_n-\frac{1}{\gamma}A(\Gamma)$. For a given graph $\Gamma=(G,\sigma)$ with $n$ vertices, let $P(\Gamma)=D(\Gamma)^{-1}A(\Gamma)$ then
\begin{equation*}
    \begin{split}
        \mathbb{L}(\Gamma)&=D(\Gamma)^{\frac{1}{2}}\big(I_n-D(\Gamma)^{-1}A(\Gamma)\big)D(\Gamma)^{-\frac{1}{2}}\\
        &=D(\Gamma)^{\frac{1}{2}}\big(I_n-P(\Gamma)\big)D(\Gamma)^{-\frac{1}{2}}
    \end{split}
\end{equation*}
The matrix $I_n-P(\Gamma)$ is called the random walk signed Laplacian~\cite{kunegis2010spectral} and is studied mainly in clustering of signed graph using normalized cuts. If the spectrum of $\mathbb{L}(\Gamma)$ and $P(\Gamma)$ are $(\lambda_1,\lambda_2,\cdots,\lambda_n)$ and $(\mu_1,\mu_2,\cdots,\mu_n)$ respectively then\\
\[
\lambda_j=1-\mu_j~~~~~\text{ for } ~j=1,2,\cdots,n
\]
For more details on normalized Laplacian spectrum refer to~\cite{chen2017normalized} and references therein.
A signed graph $\Gamma=(G,\sigma)$ can be switch to a new signed graph $\Gamma^\theta=(G,\sigma^\theta)$ by a switching function $\theta: V(\Gamma) \rightarrow \{+,-\}$. The underlying graph of $\Gamma^\theta$ is same as that of $\Gamma$ but the signature of $\Gamma^\theta$ is define on an edge $e=v_iv_j$ by $\sigma^\theta(e)=\theta(v_i)\sigma(e)\theta(v_j)$. Two signed graphs $\Gamma_1=(G,\sigma_1)$ and $\Gamma_2=(G,\sigma_2)$ with same underlying graph $G$ are switching equivalent (denoted as $\Gamma_1 \sim \Gamma_2$) if there exist a switching function $\theta$ such that $\sigma_2(e)=\sigma^\theta_1(e)$ for every edge $e$ in $G$. Switching equivalent signed graph have same adjacency spectrum and Laplacian spectrum (both signed and signless)~\cite{hou2003laplacian,hou2019signed}.

If the signed degree of all the vertices of a signed graph $\Gamma$ is equal to $k$ then $\Gamma$ is consider as a net-regular signed graph with a net-degree of $k$~\cite{nayak2016net}. In addition if $\Gamma$ is $\gamma$-regular for some integer $\gamma$ then it is consider as co-regular with co-regularity pair $(\gamma,k)$~\cite{shahul2015co}. A signed graph is balanced if all of its cycles contain even number of negative edges. Frank Harary first introduced the notion of balanced signed graph (see~\cite{harary1953notion}).

A marking $\mu: V(\Gamma)\rightarrow \{+,-\}$ is a function which assigns a sign to the vertices of the signed graph. This leads to the representation of the signed graph as a 3-tuple $\Gamma=(G,\sigma,\mu)$. This paper primarily focuses on the examination of two types of markings, namely canonical marking expressed as
\begin{math}
\mu^c(w)=\prod\limits_{e\in E_{w}}~\sigma(e)
\end{math}
where $E_w$ represents the set of edges adjacent to vertex $w$ and plurality marking expressed as\\
\[
 \mu^p(w) =  
\begin{cases}
    - &\text{ if } d^+(w) < d^-(w)\\
    + & \text{ otherwise }
\end{cases}
\]
McLema and McNicholas~\cite{mcleman2011spectra} introduced coronal of adjacency matrix for unsigned graph. Shu and Gui~\cite{cui2012spectrum} extended and generalized this concept, defining the corona for both the Laplacian matrix and signless Laplacian matrix of unsigned graphs. Later Singh et. al~\cite{singh2023structural} defined signed coronal as follows.
\textit{Consider a signed graph $\Gamma=(G,\sigma,\mu)$ with vertex set $\{v_1,v_2,\cdots,v_n\}$ and let $N$ be a graph matrix associated with $\Gamma$. When viewed as a matrix over the field of rational functions $\mathbb{C}(X)$, the characteristic matrix $\lambda I_n - N$ possesses a non-zero determinant, rendering it invertible. The signed $N$-coronal, denoted as $\Sigma_N(X) \in \mathbb{C}(X)$ for $\Gamma$, is defined by the expression}
\begin{equation}
\label{eqn 1.1}
\Sigma_N(X)=\mu(\Gamma)^T(X I_n-N)^{-1} \mu(\Gamma)
\end{equation}
where $\mu(\Gamma)=[\mu(v_1),\mu(v_2),\cdots,\mu(v_n)]^T$ and 
\[
 \mu(v_i) =  
\begin{cases}
    +1 &\text{ if marking of } v_i \text{ is } +\\
    -1 &\text{ if marking of } v_i \text{ is } -
\end{cases}
\]
In the work of Adhikari et al. \cite{adhikari2023corona}, structural properties of the corona product of two signed graphs were established. Cui and Tian \cite{cui2012spectrum} provided the adjacency spectrum and signless Laplacian spectrum of $G\diamond H$ where $G$ is regular. Liu and Lu \cite{liu2013spectra} provided the adjacency spectrum and Laplacian spectrum of the subdivision vertex neighbourhood corona ($G\boxdot H$) and subdivision edge neighborhood corona ($G\boxminus H$) for unsigned graphs $G$ and $H$ where $G$ is regular. Furthermore, Chen and Liao \cite{chen2017normalized} conducted an analysis of the normalized Laplacian spectrum for the edge corona of two unsigned graphs $G$ and $H$ where $H$ is regular. Also, Das and Panigrahi \cite{das2017normalized} conducted an analysis of the normalized Laplacian spectrum for sub-division vertex and sub-division edge neighborhood corona of two unsigned regular graphs. However, the above work has not been carried out for signed graphs.

In Section \ref{rorientation}, we first introduce a new concept called $r$-orientation and use it to define edge corona, subdivision vertex and subdivision edge neighbourhood corona of signed graphs. In Section \ref{structural}, we analyze the structural properties of edge corona, subdivision vertex and subdivision edge neighbourhood corona of signed graphs defined under $r$-orientation. Section \ref{edge} comprises of adjacency spectrum and Laplacian spectrum of edge corona of $\Gamma^1$ (regular) and $\Gamma^2$ under $r$-orientation of edges of $\Gamma^1$. Section \ref{subdivision} comprises of adjacency spectrum and Laplacian spectrum of subdivision vertex and edge neighbourhood corona of $\Gamma^1$ (regular) and $\Gamma^2$ under $r$-orientation of edges of $\Gamma^1$. Section \ref{normalized} contains the normalized Laplacian spectrum of edge corona of $\Gamma^1$ and $\Gamma^2$ (regular), subdivision vertex neighbourhood corona of $\Gamma^1$ and $\Gamma^2$ (both regular) and subdivision edge neighbourhood corona of $\Gamma^1$ and $\Gamma^2$ (regular), all defined under $r$-orientation.
\subsection{Notations and Result used}
Let $\Gamma^s=(G_s,\sigma_s,\mu_s)$ be a signed graph on $n$ vertices and $m$ edges.
\begin{enumerate}[label=(\Roman*)]
    \item $\lambda_j(M(\Gamma^s))$ denotes the eigenvalues of $M$-matrix of $\Gamma^s$ for $j=1,2,\cdots,n$.
    \item $f_{M(\Gamma^s)}(\lambda)$ denotes the characteristic polynomial of $M$-matrix of $\Gamma^s$.
    \item \textbf{Schur's Lemma:} \label{schur}\cite{bapat2010graphs} Let $C=\begin{bmatrix}
        C_{11} && C_{12}\\
        C_{21} && C_{22}
    \end{bmatrix}$ be an $n\times n$ matrix where $C_{11}$ and $C_{22}$ are square matrices.\\
    (i) If $C_{11}$ is non-singular then $det(C)=det(C_{11}). det(C_{22}-C_{21}C_{11}^{-1}C_{12})$.\\
    (ii) If $C_{22}$ is non-singular then $det(C)=det(C_{22}). det(C_{11}-C_{12}C_{22}^{-1}C_{21})$.\\ \\
    where $C_{22}-C_{21}C_{11}^{-1}C_{12}$ and $C_{11}-C_{12}C_{22}^{-1}C_{21}$ are the Schur complements of $C_{11}$ and $C_{22}$ respectively.
    \item The Kronecker (or tensor) product of two matrices, denoted as $X \otimes Y$, is formed by taking all possible products of elements from matrix $X$ with elements from matrix $Y$ and arranging them in a block matrix format. When multiplying two Kronecker products, $A \otimes B$ and $C \otimes D$, the result is $(A C) \otimes (B D)$, as long as the individual matrix products $AC$ and $BD$ exist.
 \cite{pollock1985tensor}.
\end{enumerate}
\section[Short Title for TOC and Headers]{$r$-orientation of signed graphs}\label{rorientation}
Orientation of edges of a signed graph is already defined and various results on this topic has already been established~\cite{grossman1994algebraic,zaslavsky1991orientation, zaslavsky1992orientation}. Similar to that of orientation of edges defined in ~\cite{belardo2015laplacian}, we defined $r$-orientation of edges of a signed graph but by changing the directions of arrows of edges as shown in Figure \ref{fig 1}.\\
\begin{figure}[ht]
    \centering
    \includegraphics[scale=1.3]{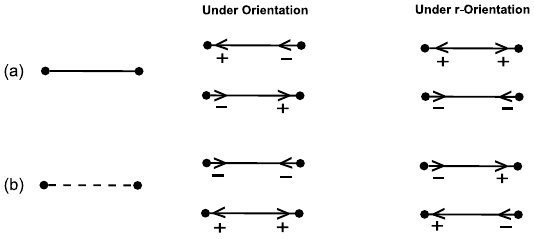}
    \caption{orientation and $r$-orientation of (a) positive and (b) negative edges.}
    \label{fig 1}
\end{figure}
 Let $\Gamma=(G,\sigma)$ be a signed graph. An $r$-oriented signed graph is a bi-directed graph where each edge is assigned with two arrows. The sign of a edge is negative if and only if both the arrows are pointed toward same direction. Basically, an $r$-oriented signed graph is an ordered pair $\Gamma_\theta=(\Gamma,\theta)$ where $\theta$ is an $r$-orientation of edges of $\Gamma$ given by\\
    \[
    \theta:V(G)\times E(G)\rightarrow \{-1,1,0\}
    \]
    which satisfy\\
    (i) $\theta(w,uv)=0$ whenever $w\neq u,v$\\
    (ii) $\theta(u,uv)=+1$(or -1) if an arrow at $u$ is going into (respectively out of) $u$\\
    (iii) $\theta(u,uv)\theta(v,uv)=\sigma(vw)$.\\
    For a signed graph $\Gamma=(\Gamma,\sigma)$ with vertex set $V(\Gamma)=\{u_1,u_2,\cdots,u_n\}$ and edge set $E(\Gamma)=\{e_1,e_2,\cdots,e_m\}$, the adjacency matrix of $\Gamma_\theta$ is given by $A(\Gamma_\theta)=(a_{lk})_{n\times n}$ where $a_{lk}=\theta(u_l,u_lu_k)\theta(u_k,u_lu_k)$ and the vertex-edge incidence matrix of $\Gamma_\theta$ is $n\times m$ matrix given by $R(\Gamma_\theta)=(b_{ij})$ where $b_{ij}=\theta(u_i,e_j)$. It is clear that $A(\Gamma_\theta)=A(\Gamma)$ for any $r$-orientation $\theta$ of edges of $\Gamma$. A line signed graph $\mathcal{L}(\Gamma_\theta)$ of an $r$-oriented signed graph $\Gamma_\theta$ is a signed graph in which the edges of $\Gamma$ has one to one correspondence to the vertices of $\mathcal{L}(\Gamma_\theta)$ and the signed of the edge $ab$ in $\mathcal{L}(\Gamma_\theta)$ ($a,b \in E(\Gamma)$) is equal to $\sigma_{\mathcal{L}}(ab)=\theta(x,a)\theta(x,b)$ where $x$ is a common vertex of edges $a$ and $b$ in $\Gamma$. The subdivision signed graph $S(\Gamma_\theta)$ of $\Gamma_\theta$ is the signed graph obtained by inserting a new vertex $v_{e_j}$ into edge $e_j(j=1,2,\cdots,m)$ of $\Gamma_\theta$. The sign of edges in $S(\Gamma_\theta)$ is given by $\sigma_S(v_iv_{e_j})=\theta(v_i,e_j)$. It is important to note that $V(S(\Gamma_\theta))=I(\Gamma_\theta) \cup V(\Gamma)$ where $I(\Gamma_\theta)$ is the set of inserted vertices of $S(\Gamma_\theta)$ that is $|I(\Gamma_\theta)|=|E(\Gamma)|$ and the edges in $S(\Gamma_\theta)$ is represented by $v_iv_{e_j}$ where $v_i$ is an end point of edge $e_j$ in $\Gamma_\theta$.
    \begin{figure}[ht]
    \centering
    \includegraphics[scale=1.13]{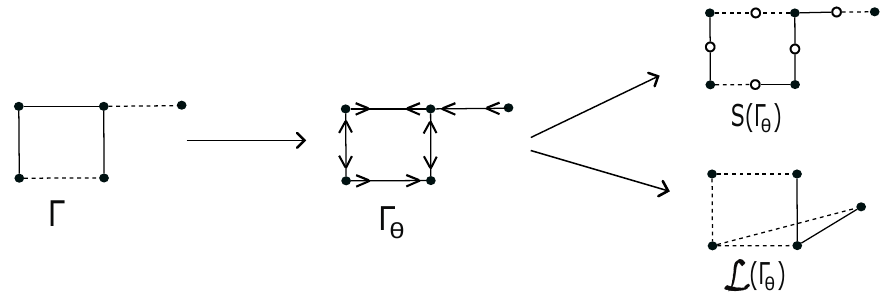}
    \caption{A line signed graph $\mathcal{L}(\Gamma_\theta)$ and a subdivision signed graph $S(\Gamma_\theta)$ of a signed graph $\Gamma$ under $r$-orientation $\theta$ of edges of $\Gamma$.}
    \label{fig 2}
\end{figure}
    \begin{defi}\label{defi 1.1} Suppose $\Gamma^i=(G_i,\sigma_i,\mu_i)$ be signed graphs on $n_i$ vertices, $m_i$ edges for $i=1,2$ and $\theta$ be any $r$-orientation of edges of $\Gamma^1$. The edge corona product of $\Gamma^1$ and $\Gamma^2$ under $r$-orientation $\theta$, denoted by $(\Gamma^1 \diamond  \Gamma^2)_\theta$, is a signed graph obtained by taking one copy of $\Gamma^1$ and $m_1$ copies of $\Gamma^2$ and then forming signed edges by joining two end vertices of the $t^{th}$ edge (say $e_t$) of $\Gamma^1$ to every vertex in the $t^{th}$ copy of $\Gamma^2$. The sign of the new edge formed by joining an end vertex $u$ of the $e_t$ and $j^{th}$ vertex in $t^{th}$ copy of $\Gamma^2$ say $w$ is given by $\theta(u,e_t)\mu_1(w)$.
    \end{defi}
 We can partition the vertices of edge corona of $\Gamma^1$ (with $n_1$ vertices and $m_2$ edges) and $\Gamma^2$ (with $n_2$ vertices and $m_2$ edges) under $r$-orientation $\theta$ of edges of $\Gamma$ as follows:\\
  Let $V(\Gamma^1)=\{u_1,u_2,\cdots,u_{n_1}\}$, $E(\Gamma^1)=\{e_1,e_2,\cdots,e_{m_1}\}$ and  $V^l(\Gamma^2)=\{w^l_1,w^l_2,\cdots,w^l_{n_2}\}$ denotes vertex set of $l^{th}$ copy of $\Gamma^2$ then,
      \begin{equation}
      \label{eqn 1.2}
     V(\Gamma^1) \cup [V^1(\Gamma^2)\cup V^2(\Gamma^2)\cdots V^m(\Gamma^2) ]
     \end{equation}
     is the partition of $V(\Gamma^1 \diamond \Gamma^2)_\theta$. Clearly the degree of the vertices of $(\Gamma^1 \diamond \Gamma^2)_\theta$ are
     \begin{equation} \label{eqn 1.3}
     \begin{split}
d_{(\Gamma^1 \diamond \Gamma^2)_\theta}(u_j)&=d_{\Gamma^1}(u_j)(1+n_2);~~ j=1,2,\cdots,n_1\\
             d_{(\Gamma^1 \diamond \Gamma^2)_\theta}(w^l_k)&=d_{\Gamma^2}(w_k)+2; ~~ l=1,2,\cdots,m_1; ~ k=1,2,\cdots,n_2
             \end{split}
     \end{equation}
    \begin{defi}\label{defi 1.2} Suppose $\Gamma^i=(G_i,\sigma_i,\mu_i)$ be signed graphs on $n_i$ vertices, $m_i$ edges for $i=1,2$ and $\theta$ be any $r$-orientation of edges of $\Gamma^1$. The subdivision vertex neighbourhood corona of $\Gamma^1$ and $\Gamma^2$ under an $r$-orientation $\theta$, denoted by $(\Gamma^1 \boxdot \Gamma^2)_\theta$, is the signed graph obtained from one copy of $S(\Gamma^1_\theta)$ and $n_1$ disjoint copies of $\Gamma^2$ and then connecting neighbours of the $t^{th}$ vertex of $\Gamma^1$ in $S(\Gamma^1_\theta)$ to every vertex within the $t^{th}$ copy of $\Gamma^2$. The sign of the newly introduced edge between a neighbour of $t^{th}$ vertex of $\Gamma^1$ in $S(\Gamma^1_\theta)$ say $v_e$ and $k^{th}$ vertex in the $t^{th}$ copy of $\Gamma^2$ say $u$ is given by $\theta(v_t,e)\mu_2(u)$ where $v_t$ is the $t^{th}$ vertex of $\Gamma^1$. The subdivision edge neighbourhood corona of $\Gamma^1$ and $\Gamma^2$ under an $r$-orientation $\theta$, denoted by $(\Gamma^1 \boxminus \Gamma^2)_\theta$, is the signed graph obtained by taking one copy of $S(\Gamma^1_\theta)$ and $|I(\Gamma^1_\theta)|$ disjoint copies of $\Gamma^2$ and then connecting the neighbours of the $t^{th}$ vertex of $I(\Gamma^1_\theta)$ to every vertex within the $t^{th}$ copy of $\Gamma^2$. The sign of newly introduced edge between a neighbour of the $t^{th}$ vertex of $I(\Gamma^1_\theta)$ in $S(\Gamma^1_\theta)$  say $v$ and $k^{th}$ vertex in the $t^{th}$ copy of $\Gamma^2$ say $u$ is given by $\theta(v,e)\mu_2(u)$ where $v_e$ is the $t^{th}$ vertex of $I(\Gamma^1_\theta)$.
     \end{defi}
     We can partition the vertices of subdivision vertex neighbourhood corona and subdivision edge neighbourhood corona of $\Gamma^1$ ($n_1$ vertices and $m_1$ edges) and $\Gamma^2$ ($n_2$ vertices and $m_2$ edges) under an $r$-orientation $\theta$ of edges of $\Gamma^1$ as follows:\\
     Let $V(\Gamma^1)=\{v_1,v_2,\cdots,v_{n_1}\}$, $I(\Gamma^1_\theta)=\{v_{e_1},v_{e_2},\cdots,v_{e_{m_1}}\}$ and $V(\Gamma^2)=\{u_1,u_2,\cdots,u_{n_2}\}$. Let $V^k(\Gamma^2)=\{u^k_1,u^k_2,\cdots,u^k_{n_2}\}$ denote the vertex set of the $k^{th}$ copy of $\Gamma^2$ then,
     \begin{equation}
     \label{eqn 1.4}
     ~~~~~~V(\Gamma^1) \cup I(\Gamma^1_\theta) \cup [V^1(\Gamma^2)\cup V^2(\Gamma^2)\cdots V^{n_1}(\Gamma^2) ]
     \end{equation}
     \begin{equation}
     \label{eqn 1.5}
     \text{ and } V(\Gamma^1) \cup I(\Gamma^1_\theta) \cup [V^1(\Gamma^2)\cup V^2(\Gamma^2)\cdots V^{m_1}(\Gamma^2) ]
     \end{equation}
     are the partition of $V(\Gamma^1 \boxdot \Gamma^2)_\theta$ and $V(\Gamma^1 \boxminus \Gamma^2)_\theta$ respectively. Clearly the degree of the vertices of $(\Gamma^1 \boxdot \Gamma^2)_\theta$ and $(\Gamma^1 \boxminus \Gamma^2)_\theta$ are
     \begin{align*}
             &d_{(\Gamma^1 \boxdot \Gamma^2)_\theta}(v_j)=d_{\Gamma^1}(v_j);& 
&d_{(\Gamma^1 \boxminus \Gamma^2)_\theta}(v_j)=d_{\Gamma^1}(v_j)(1+n_2)&j=1,2,\cdots,n_1\\
             &d_{(\Gamma^1 \boxdot \Gamma^2)_\theta}(v_{e_k})=2+2n_2;& &d_{(\Gamma^1 \boxminus \Gamma^2)_\theta}(v_{e_k})=2&k=1,2,\cdots,m_1\\
             &d_{(\Gamma^1 \boxdot \Gamma^2)_\theta}(u^
             j_l)=d_{\Gamma^2}(u_l)+d_{\Gamma^1}(v_j);& &d_{(\Gamma^1 \boxminus \Gamma^2)_\theta}(u^j_l)=d_{\Gamma^2}(u_l)+2&j=1,2,\cdots,n_1;l=1,2,\cdots,n_2
     \end{align*}
    \begin{figure}[ht]
        \centering
\includegraphics{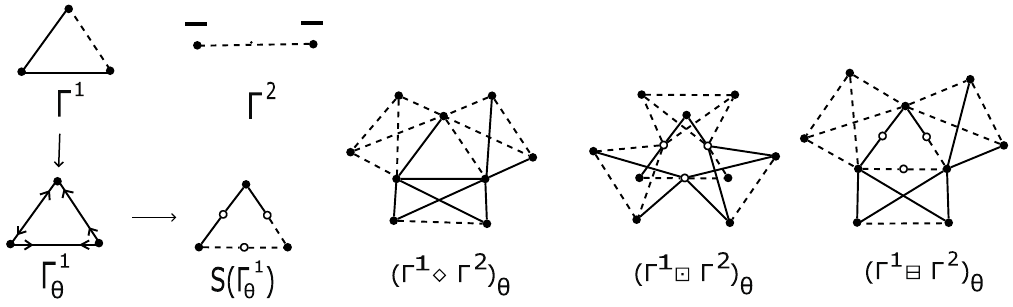}
        \caption{Subdivision graph, edge corona, subdivision vertex neighbourhood and subdivision edge neighbourhood corona of signed graphs under $r$-orientation.}
        \label{fig:my_label}
    \end{figure}
 \begin{lem} \label{lem 1.3} Let $\Gamma=(G,\sigma)$ be a signed graph on $n$ vertices and $m$ edges and $\theta$ be any $r$-orientation of edges of $\Gamma$ then $R(\Gamma_\theta)R(\Gamma_\theta)^T=Q(\Gamma)$. 
 \end{lem}
     \begin{proof} The rows of $R(\Gamma_\theta)$ are indexed by $V(\Gamma)$. Thus the $(k,j)$-entry of $R(\Gamma_\theta)R(\Gamma_\theta)^T$ is the inner product of the rows $k$ and $j$ of $R(\Gamma_\theta)$. If $k=j$ then the inner product is $d_k$ and if $k \neq j$ then the inner product is $\theta(k,kj)\theta(j,kj)=\sigma(kj)$. Hence $R(\Gamma_\theta)R(\Gamma_\theta)^T=Q(\Gamma)$.
     \end{proof}
     \noindent If $\Gamma$ is $\gamma$-regular signed graph then $R(\Gamma_\theta)R(\Gamma_\theta)^T=\gamma I_n+A(\Gamma)$.
     \begin{lem}\label{lem 1.4} Let $\Gamma=(G,\sigma)$ be a signed graph on $n$ vertices and $m$ edges and $\theta$ be any $r$-orientation of edges of $\Gamma$ then $R(\Gamma_\theta)^TR(\Gamma_\theta)=2I_m+A(\mathcal{L}(\Gamma_\theta))$.
     \end{lem}
     \begin{proof}
      Suppose $E(\Gamma)=\{e_1,e_2,\cdots,e_m\}$. The columns of $R(\Gamma_\theta)$ are indexed by $E(\Gamma)$. Thus the $(k,j)$-entry of $R(\Gamma_\theta)^TR(\Gamma_\theta)$ is the inner product of the columns $k$ and $j$ of $R(\Gamma_\theta)$. If $k=j$ then the inner product is $2$, since each edge has two end points. If $k \neq j$ and $e_k$ and $e_j$ are not adjacent then the inner product is 0 and if $k \neq j$ and $e_k$ and $e_j$ are adjacent then the inner product is $\theta(w,e_k)\theta(w,e_j)$ where $w$ is a common vertex of $e_k$ and $e_j$.  Hence $R(\Gamma_\theta)^TR(\Gamma_\theta)=2I_m+A(\mathcal{L}(\Gamma_\theta))$.
      \end{proof}
     \begin{lem}\label{lem 1.5} Let $\theta$ and $\theta'$ be two different $r$-orientation of edges of signed graph $\Gamma=(G,\sigma)$ then $A(\mathcal{L}(\Gamma_\theta))\sim A(\mathcal{L}(\Gamma_{\theta'}))$ and $A(S(\Gamma_\theta))\sim A(S(\Gamma_{\theta'}))$.
     \end{lem}
     \begin{proof} Let $E(\Gamma)=\{e_1,e_2,\cdots,e_m\}$. Without loss of generality we can assume edges $e_1,e_2,\cdots,e_k$; $1\leqslant k \leqslant m$ has different $r$-orientation in $\Gamma_\theta$ and $\Gamma_{\theta'}$. Suppose \begin{math}
         S=\begin{bmatrix}
             -I_k&& 0\\
             0&& I_{m-k}
         \end{bmatrix}
     \end{math}. Then $R(\Gamma_{\theta'})=R(\Gamma_\theta)S$
     \begin{equation*}
         \begin{split}
              2I_m+A(\mathcal{L}(\Gamma_{\theta'}))&=R(\Gamma_{\theta'})^TR(\Gamma_{\theta'})\\
             &=SR(\Gamma_\theta)^TR(\Gamma_\theta)S\\
             &=S(2I_m+A(\mathcal{L}(\Gamma_{\theta})))S\\
             &=2I_m+SA(\mathcal{L}(\Gamma_{\theta}))S\\
         \end{split}
     \end{equation*}
     $\therefore A(\mathcal{L}(\Gamma_{\theta'}))=SA(\mathcal{L}(\Gamma_{\theta}))S$. Thus $A(\mathcal{L}(\Gamma_\theta))\sim A(\mathcal{L}(\Gamma_{\theta'}))$\\
     Taking $S'=I+S$ we get $A(S(\Gamma_{\theta'}))=S'A(S(\Gamma_\theta))S'$. Thus $A(S(\Gamma_\theta))\sim A(S(\Gamma_{\theta'}))$.
     \end{proof}
     \begin{lem} \cite{zaslavsky2013matrices} Consider two signed graphs $\Gamma_1=(G,\sigma_1)$ and $\Gamma_2=(G,\sigma_2)$ on same underlying graph $G$. $\Gamma_1\sim \Gamma_2$ if and only if $A(\Gamma_1)\sim A(\Gamma_2)$.
     \end{lem}
     \begin{rem}\label{rem 1.7} For any two $r$-orientations $\theta$ and $\theta'$ of $\Gamma=(G,\sigma)$, $\mathcal{L}(\Gamma_\theta)\sim \mathcal{L}(\Gamma_{\theta'})$ and $S(\Gamma_\theta) \sim S(\Gamma_{\theta'})$.
     \end{rem}
     \begin{lem}\label{lem 1.8} Let $\Gamma=(G,\sigma)$ be $\gamma$-regular  signed graph on $n$ vertices and $m$ edges and $\theta$ be any $r$-orientation of edges of $\Gamma$. If the eigenvalues of $A(\Gamma)$ are $\mu_1,\mu_2,\cdots,\mu_n$ then the eigenvalues of $A(\mathcal{L}(\Gamma_\theta))$ are $\mu_j+\gamma-2,j=1,2,\cdots,n$ and $-2$ with multiplicity $m-n$.
     \end{lem}
     \begin{proof} Since $\Gamma$ is $\gamma$-regular, by Lemma \ref{lem 1.3} $R(\Gamma_\theta)R(\Gamma_\theta)^T=\gamma I+A(\Gamma)$ and by Lemma \ref{lem 1.4} $R(\Gamma_\theta)^TR(\Gamma_\theta)=2I_m+A(\mathcal{L}(\Gamma_\theta))$. Thus the eigenvalues of $R(\Gamma_\theta)R(\Gamma_\theta)^T$ are $\gamma+\lambda_1,\gamma+\lambda_2,\cdots,\gamma+\lambda_n$ and so the eigenvalues of $A(\mathcal{L}(\Gamma_\theta))$ are $\gamma+\mu_1-2,\gamma+\mu_2-2,\cdots,\gamma+\mu_n-2$ and $-2$ with multiplicity $m-n$.
     \end{proof}
     \begin{lem}\label{lem 1.9} Let $\Gamma=(K_{1,m},\sigma,\mu)$ be a signed star with $V(\Gamma)=\{v_1,v_2,\cdots,v_{m+1}\}$ such that $d(v_1)=m$ and $\mu=\mu^p$ or $\mu^c$ then
 \begin{align*}
 (i)~\Sigma_{A(\Gamma)}(X)&=\frac{(m+1)X+2m\mu(v_1)}{X^2-m}\\
  (ii)~\Sigma_{L(\Gamma)}(X)&=\frac{(m+1)X-(m^2+1)-2m\mu(v_1)}{X\left(X-(m+1)\right)}\\
(iii) ~\Sigma_{Q(\Gamma)}(X)&=\frac{(m+1)X-(m^2+1)+2m\mu(v_1)}{X(X-m-1)}
 \end{align*}
 \end{lem}
 \begin{proof} $(i)$ Here $A(\Gamma)=$\begin{math}
 \begin{bmatrix}
           0 &\mu(v_2) &\cdots &\mu(v_{m+1}) \\
           \mu(v_2)   &0 &\cdots &0 \\
           \vdots &\vdots &  &\vdots\\
           \mu(v_{m+1})   &0 &\cdots &0
         \end{bmatrix}
\end{math}
. Let $\tau=diag\Big(\frac{m+X \mu(v_1)}{\mu(v_1)}, X+\mu(v_1), \cdots, X+\mu(v_1)\Big)$ be $(m+1) \times (m+1)$ diagonal matrix with first diagonal entry as $\frac{m+X \mu(v_1)}{\mu(v_1)}$ and remaining $m$ diagonal entries as $X+\mu(v_1)$. Then
\begin{math}
\left(X I_{m+1} - A(\Gamma)\right)\tau\mu(\Gamma) = (X^2 - n)~\mu(\Gamma)
\end{math}\\
Thus,
\begin{align*}
    \Sigma_{A(\Gamma)}(X)=&\mu(\Gamma)^T~\left(X~I_m - A(\Gamma)\right)^{-1}~\mu(\Gamma)\\
     =&\frac{\mu(\Gamma)^T\tau\mu(\Gamma)}{X^2 - m}\\
     =&\frac{(m+1)X+2m\mu(v_1)}{X^2 - m}
\end{align*}
$(ii)$ and $(iii)$ can be proved similarly.
\end{proof}
\begin{lem}\label{lem 1.10}~\cite{singh2023structural} Let $\Gamma=(G,\sigma,\mu)$ be co-regular graph of order $m$ and marking $\mu=\mu^c$ or $\mu^p$ with co-regularity pair $(\gamma,f)$ then
\begin{align*}
        (i)~\Sigma_{A(\Gamma)}(X)&=\frac{m}{X-f}
        && (ii)~\Sigma_{L(\Gamma)}(X)&=\frac{m}{X-\gamma+f}
        && (iii)~\Sigma_{Q(\Gamma)}(X)&=\frac{m}{X-\gamma-f}
\end{align*}
\end{lem}
\section[Short Title for TOC and Headers]{Structural properties of $(\Gamma^1\diamond \Gamma^2)_\theta, (\Gamma^1\boxdot \Gamma^2)_\theta$ and $(\Gamma^1\boxminus \Gamma^2)_\theta$}\label{structural}
Now we consider counting edges and triads (or 3-cycles) in $(\Gamma^1 \diamond \Gamma^2)_\theta$, $(\Gamma^1 \boxdot \Gamma^2)_\theta$ and $(\Gamma^1 \boxminus \Gamma^2)_\theta$.\\
Let $N_1^+$ (and $N_1^-$) denotes the number of arrows directed towards (resp. away) from the vertices in $\Gamma^1_\theta$ and $M_2^+$ (and $M_2^-$) denotes the number of positively (resp. negatively) marked vertices in $\Gamma^2$. Then $N_1^++N_1^-=2|E(\Gamma^1)|$. Suppose $s \in \{+,-\}$. We represent the quantity $|F_1^+|^{\pp}$ as the count of edges in $\Gamma^1_\theta$ with a positive sign, where both the arrows are directed towards the vertices. Similarly, $|F_1^+|^{\mm}$ represents the count of edges in $\Gamma^1_\theta$ with a positive sign, but with both arrows directed away from the vertices and $|F_1^-|^{\pm}$ denotes the quantity of negatively signed edges in $\Gamma^1_\theta$. Also we use $|E_2^s|^{\pp}$ to signify the number of edges with sign $s$ that connect two vertices marked as positive in $\Gamma^2$, $|E_2^s|^{\pm}$ represents the count of edges with sign $s$ that connect one positively marked vertex and one negatively marked vertex in $\Gamma^2$ and $|E_2^s|^{\mm}$ denotes the number of edges with sign $s$ that connect two negatively marked vertices in $\Gamma^2$. If $T_i$ denotes a number of triads having $r$ number of negative edges, $i=0,1,2,3$ then
\begin{table}[ht]
\label{table 1}
\begin{center}
\caption{Counts of edges in $(\Gamma^1 \diamond \Gamma^2)_\theta$, $(\Gamma^1 \boxdot \Gamma^2)_\theta$ and $(\Gamma^1 \boxminus \Gamma^2)_\theta$.}
\scalebox{0.8}{
\begin{tabular}{|c| c| c| c| c| c| c|}
 \hline
 &&&&&&\\
  Edges & $\Gamma^1$ & $\Gamma^2$ & $S(\Gamma^1_\theta)$ & $(\Gamma^1 \diamond \Gamma^2)_\theta$ & $(\Gamma^1 \boxdot \Gamma^2)_\theta$ & $(\Gamma^1 \boxminus \Gamma^2)_\theta$ \\
  &&&&&&\\
 \hline
 &&&&&&\\
  \# of edges & $|E_1|$ & $|E_2|$ & $2|E_1|$ & $|E_1|+2|E_1||V_2|$ & $2|E_1|+|V_1||E_2|$ & $2|E_1|+|E_1||E_2|$\\
  & & & &$+|E_1||E_2|$ & $+2|E_1||V_2|$ & +$2|E_1||V_2|$\\
  &&&&&&\\
 \hline
 &&&&&&\\
 \# of $+$ edges & $|E^+_1|$ & $|E^+_2|$ & $N_1^+$ & $|E^+_1|+|E_1||E^+_2|$ & $N_1^++|V_1||E_2^+|$ & $N_1^++|E_1||E_2^+|$ \\
 & & & &$+N^+_1M^+_2+N^-_1M^-_2$ & $+N_1^+M_2^++N_1^-M_2^-$ & $+N_1^+M_2^++N_1^-M_2^+$ \\
 &&&&&&\\
 \hline
 &&&&&&\\
 \# of $-$ edges & $|E^-_1|$ & $|E^-_2|$ & $N_1^-$ & $|E^-_1|+|E_1||E^-_2|$ & $N_1^-+|V_1||E_2^-|$ & $N_1^-+|E_1||E_2^-|$ \\
 & & & &$+N^+_1M^-_2+N^-_1M^+_2$ & $+N_1^+M_2^-+N_1^-M_2^+$ & $+N_1^+M_2^-+N_1^-M_2^+$ \\
 &&&&&&\\
 \hline
\end{tabular}}
\end{center}
\end{table}
\begin{table}[ht]
\label{table 2}
\begin{center}
\caption{Counts of triads in $(\Gamma^1 \diamond \Gamma^2)_\theta$, $(\Gamma^1 \boxdot \Gamma^2)_\theta$ and $(\Gamma^1 \boxminus \Gamma^2)_\theta$.}
\scalebox{0.73}{
\begin{tabular}{|c| c| c| c| c| c| c|}
 \hline
 &&&&&&\\
 Triads & $\Gamma^1$ & $\Gamma^2$ & $S(\Gamma_\theta^1)$ & $(\Gamma^1 \diamond \Gamma^2)_\theta$ & $(\Gamma^1 \boxdot \Gamma^2)_\theta$ & $(\Gamma^1 \boxminus \Gamma^2)_\theta$ \\
 &&&&&&\\
 \hline
 &&&&&&\\
  \# of $T_0$ & $|T_0(\Gamma^1)|$ & $|T_0(\Gamma^2)|$ & 0 & $T_0(\Gamma^1)+|E_1||T_0(\Gamma^2)|$ & $|V_1||T_0(\Gamma^2)|$ & $|E_1||T_0(\Gamma^2)|$\\
  & & & &$+N_1^+|E_2^+|^{\pp}+N_1^-|E_2^+|^{\pd}    
  $ & $+N_1^+|E_2^+|^{\pp}$ & $+N_1^+|E_2^+|^{\pp}$ \\
  & & & &$+|F_1^+|^{\pp} M_2^++|F_1^+|^{\mm}M_2^-$ & $+N_1^-|E_2^+|^{\mm}$ & $+N_1^-|E_2^+|^{\mm}$ \\
  &&&&&&\\
 \hline
 &&&&&&\\
 \# of $T_1$ & $|T_1(\Gamma^1)|$ & $|T_1(\Gamma^2)|$ & 0 & $|T_1(\Gamma^1)|+|E_1||T_1(\Gamma^2)|$ & $|V_1||T_1(\Gamma^2)|$ & $|E_1||T_1(\Gamma^2)|$ \\
 & & & &$+N_1^+|E_2^-|^{\pp}+N_1^-|E_2^-|^{\mm}$ & $+N_1^+|E_2^-|^{\pp}$ & $+N_1^+|E_2^-|^{\pp}$ \\ 
 & & & & $+2|E(\Gamma^1)||E_2^+|^{\pd}$ & $+N_1^-|E_2^-|^{\mm}+2|E(\Gamma^1)||E_2^+|^{\pd}$ & $+N_1^-|E_2^-|^{\mm}+2|E(\Gamma^1)||E_2^+|^{\pd}$ \\
 %& & & & & $2|E(\Gamma^1)||E_2^+|^{\pd}$ & %$2|E(\Gamma^1)||E_2^+|^{\pd}$\\
 &&&&&&\\
 \hline
 &&&&&&\\
  \# of $T_2$ & $|T_2(\Gamma^1)|$ & $|T_2(\Gamma^2)|$ & 0 & $|T_2(\Gamma^1)|+|E_1||T_2(\Gamma^2)|$ & $|V_1||T_2(\Gamma^2)|$ & $|E_1||T_2(\Gamma^2)|$ \\
 & & & &$+N_1^+|E_2^+|^{\mm}+M_2^-|F_1^+|^{\pp}$ & $+N_1^+|E_2^+|^{\mm}$ & $+N_1^+|E_2^+|^{\mm}$ \\ 
 & & & &$+N_1^-|E_2^+|^{\pp}+M_2^+|F_1^+|^{\mm}$ & $+N_1^-|E_2^+|^{\pp}+$ & $+N_1^-|E_2^+|^{\pp}+$ \\
 & & & & $+2|E(\Gamma^1)||E_2^-|^{\pd}+|V_2||F_1^-|^{\pd}$ & $2|E(\Gamma^1)||E_2^-|^{\pd}$ & $2|E(\Gamma^1)||E_2^-|^{\pd}$\\
 &&&&&&\\
 \hline
 &&&&&&\\
  \# of $T_3$ & $|T_3(\Gamma^1)|$ & $|T_3(\Gamma^2)|$ & 0 & $|T_3(\Gamma^1)|+|E_1||T_3(\Gamma^2)|$ & $|V_1||T_3(\Gamma^2)|$ & $|E_1||T_3(\Gamma^2)|$ \\
 & & & &$+N_1^+|E_2^-|^{\mm}+N_1^-|E_2^-|^{\pp}$ & $+N_1^+|E_2^-|^{\mm}+N_1^-|E_2^-|^{\pp}$ & $+N_1^+|E_2^-|^{\mm}+N_1^-|E_2^-|^{\pp}$ \\ 
 %& & & & & $+N_1^-|E_2^-|^{\pp}$ & $+N_1^-%|E_2^-|^{\pp}$ \\
 &&&&&&\\
 \hline
\end{tabular}}
\end{center}
\end{table}\\
\newpage
\begin{thm} Let $\Gamma^1$ and $\Gamma^2$ be two balanced signed graphs and $\theta$ be any $r$-orientation of edges of $\Gamma^1$. Then $(\Gamma^1 \star \Gamma^2)_\theta$ is unbalanced if and only if $\Gamma^2$ includes one of the following categories of edges:\\
(a) A positive edge that connects two vertices with opposite markings.\\
(b) A negative edge that connects two vertices marked as positive.\\
(c) A negative edge that connects two vertices marked as negative.\\
Here $\star$ represents $\diamond$, $\boxdot$ and $\boxminus$.
\end{thm}
\begin{proof} From Table \ref{table 2} it is clear that a positively $r$-oriented vertex with respect to an edge in $\Gamma^1_\theta$ will form triad(s) $T_1$ in $(\Gamma^1 \star \Gamma^2)_\theta$ if there is an edge of type (a) and/or (b) otherwise it will form a triad $T_3$ if there is an edge of type (c). Similarly a negatively $r$-oriented vertex with respect to an edge in $\Gamma^1_\theta$ will form triad(s) $T_1$ if there is an edge of type (a) and/or (c) otherwise it will form a triad $T_3$ if there is an edge of type (b).
\end{proof}
If $\Gamma^1=(G,+)$ and $\Gamma^2=(H,+,+)$ then $|E_2^-|^{\pp}$, $|E_2^-|^{\mm}$, $|E_2^+|^{\pd}$ are all zero. Hence $(\Gamma^1\star \Gamma^2)_\theta\sim G\star H$
\section[Short Title for TOC and Headers]{Spectrum and Laplacian spectrum of $(\Gamma^1\diamond \Gamma^2)_\theta$} \label{edge}
\begin{thm}\label{thm 3.1} Consider $\gamma_1$-regular signed graph $\Gamma^1=(G,\sigma_1,\mu_1)$ on $n_1$ vertices and $m_1$ edges. Let $\theta$ be any $r$-orientation of edges of $\Gamma^1$ and $\Gamma^2=(H,\sigma_2,\mu_2)$ be any arbitrary signed graph on $n_2$ vertices then the characteristic polynomial of $A(\Gamma^1 \diamond \Gamma^2)_\theta$ is given by
 \begin{equation*}
 f_{A(\Gamma^1 \diamond \Gamma^2)_\theta}(\lambda)=\left(f_{A(\Gamma^2)}(\lambda)\right)^{m_{_1}}\prod_{j=1}^{n_1}\left[\lambda-\lambda_j(\Gamma^1)-\big(\gamma_1+\lambda_j(\Gamma^1)\big) \Sigma_{A(\Gamma^2)}(\lambda)\right]
 \end{equation*}
 \end{thm}
 \begin{proof} If we consider $R(\Gamma^1_\theta)$ as the vertex-edge incidence matrix of $\Gamma^1_\theta$, then with respect to the partition \ref{eqn 1.2}, the adjacency matrix of $(\Gamma_1 \diamond \Gamma_2)_\theta$ is given by
   \begin{equation*}
   \begin{split}
 A(\Gamma^1 \diamond \Gamma^2)_\theta&=
 \begin{bmatrix}
           A(\Gamma^1) &R(\Gamma^1_\theta) \otimes \mu(\Gamma_2)^T \\ \\
            R(\Gamma^1_\theta)^T \otimes \mu(\Gamma^2)  &I_{m_{_1}} \otimes A(\Gamma^2) \\
         \end{bmatrix}\\ \\
      \therefore  f_{A(\Gamma^1 \diamond \Gamma^2)_\theta}(\lambda)&=det\left(\lambda I_{n_{_1}+m_{_1}n_{_2}}-A(\Gamma^1 \diamond \Gamma^2)_\theta\right)\\ \\
       &=det\begin{bmatrix}
                  \lambda I_{n_{_1}}-A(\Gamma^1) &&-R(\Gamma^1_\theta) \otimes\mu(\Gamma^2)^T\\ \\
                  -R(\Gamma^1_\theta)^T \otimes \mu(\Gamma^2) && I_{m_{_1}}\otimes \big(\lambda I_{n_{_2}}-A(\Gamma^2)\big)\\
        \end{bmatrix}\\ \\
        &=det\left(I_{m_{_1}}\otimes \left(\lambda I_{n_{_2}}-A(\Gamma^2)\right)\right)~det\Big[\lambda I_{n_{_1}}-A(\Gamma^1)-\left(R(\Gamma^1_\theta)\otimes \mu(\Gamma^2)^T\right)\\
        &~~~~~~~~~~~~~~~~~~~~~~~~~~~~~~~~~~~~~~~~~~~~~~~~~\left(I_{m_{_1}}\otimes\left(\lambda I_{n_{_2}}-A(\Gamma^2)\right)\right)^{-1}\left(R(\Gamma^1_\theta)^T \otimes \mu(\Gamma^2)\right)\Big] \\
      &=det\left(I_{m_{_1}}\otimes \left(\lambda I_{n_{_2}}-A(\Gamma^2)\right)\right)~det\Big[\lambda I_{n_{_1}}-A(\Gamma^1)-\left(R(\Gamma^1_\theta)I_{m_{_1}}R(\Gamma^1_\theta)^T\right)\\
      &~~~~~~~~~~~~~~~~~~~~~~~~~~~~~~~~~~~~~~~~~~~~~~~~~~~~~~~~~~~~\otimes \left(\mu(\Gamma^2)^T~(\lambda I_{n_{_2}}-A(\Gamma^2))^{-1}~\mu(\Gamma^2)\right)\Big]\\
      &=det\Big(I_{m_{_1}}\otimes \big(\lambda I_{n_{_2}}-A(\Gamma^2)\big)\Big)~det\Big[\lambda I_{n_{_1}}-A(\Gamma^1)-Q(\Gamma^1) \otimes \Sigma_{A(\Gamma^2)}(\lambda)\Big]\\
      &=det\Big(I_{m_{_1}}\otimes \big(\lambda I_{n_{_2}}-A(\Gamma^2)\big)\Big)~det\Big[\lambda I_{n_{_1}}-A(\Gamma^1)- \Sigma_{A(\Gamma^2)}(\lambda)~Q(\Gamma^1)\Big]
      \end{split}
      \end{equation*}
      By Lemma \ref{lem 1.3} $R(\Gamma^1_\theta)R(\Gamma^1_\theta)^T=Q(\Gamma^1)$ for any $r$-orientation $\theta$ of edges of $\Gamma^1$. Also as $\Gamma^1$ is $\gamma_1$-regular $\lambda_j(Q(\Gamma^1))=\gamma_1+\lambda_j(A(\Gamma^1));~j=1,2,\cdots,n_1$. Thus the characteristic polynomial of $A(\Gamma^1 \diamond \Gamma^2)_\theta$ is independent of $\theta$ and is given by
      \begin{equation*}
 \therefore f_{A(\Gamma^1 \diamond \Gamma^2)_\theta}(\lambda)=\left(f_{A(\Gamma^2)}(\lambda)\right)^{m_{_1}}\prod_{j=1}^{n_1}\left[\lambda-\lambda_j(\Gamma^1)-\big(\gamma_1+\lambda_j(\Gamma^1)\big) \Sigma_{A(\Gamma^2)}(\lambda)\right] \qedhere
 \end{equation*}
    \end{proof}
    \begin{prop} Consider $\gamma_1$-regular signed graph $\Gamma^1=(G,\sigma_1,\mu_1)$ on $n_1$ vertices and $m_1$ edges. Let $\theta$ be any $r$-orientation of edges of $\Gamma^1$ and $\Gamma^2=(K_{1,n_2},\sigma_2,\mu_2)$ be a signed star with $V(\Gamma^2)=\{v_1,v_1,\cdots,v_{n_2+1}\}$ where $d(v_1)=n_2$ and $\mu_2=\mu^p$ or $\mu^c$. Suppose the spectrum of $\Gamma^1$ is $(\alpha_1,\alpha_2,\cdots,\alpha_{n_{1}})$, then the spectrum of $(\Gamma^1 \diamond \Gamma_2)_\theta$ consists of\\
    (i) $0$ with multiplicity $m_1(n_2-1)$.\\
    (ii) The roots of the equation\\ $x^3-\alpha_ix^2-\big(n_2+\gamma_1(n_2+1)+\alpha_j(n_2+1)\big)x+\alpha_jn_2-2\gamma_1n_2\mu_2(v_1)-2\alpha_jn_2\mu_2(v_1)=0$ corresponding to each eigenvalue $\alpha_j$ of $\Gamma^1$.\\
    (iii) The eigenvalues $\sqrt{n_2}$ and $-\sqrt{n_2}$ each with multiplicity $m_1-n_1$.
    \end{prop}
    \begin{proof} The spectrum of $\Gamma^2$ is $(0^{(n_2-1)},\pm\sqrt{n_2})$. By Lemma \ref{lem 1.9}(i)
    \[
    \Sigma_{A(\Gamma^2)}(\lambda)=\frac{(n_2+1)\lambda+2n_2\mu_2(v_1)}{\lambda^2-n_2}
    \]
    Two poles of $\Sigma_{A(\Gamma^2)}(\lambda)$ are $\lambda=\pm \sqrt{n_2}$. By Theorem \ref{thm 3.1}, the spectrum of $(\Gamma^1 \diamond \Gamma^2)_\theta$ is given by
    \begin{itemize}
        \item $0$ with multiplicity $m_1(n_2-1)$.
        \item The roots of the equation
        \begin{equation*}
            \begin{split}
                &\lambda-\alpha_j-(\gamma_1+\alpha_j)\Sigma_{A(\Gamma^2)}(\lambda)=0\\
                i.e~~ &\lambda^3-\alpha_j\lambda^2-\big(n_2+(\gamma_1+\alpha_j)(n_2+1)\big)\lambda+\alpha_jn_2-2n_2\mu_2(v_1)(\gamma_1+\alpha_j)=0\\
                &\text{ corresponding to each eigenvalue } \alpha_j(j=1,2,\cdots,n_1) \text{ of } \Gamma^1.
            \end{split}
        \end{equation*}
    \end{itemize}
    The remaining $n_1+m_1(n_2+1)-(m_1(n_2-1)+3n_1)=2(m_1-n_1)$ eigenvalues of $(\Gamma^1 \diamond \Gamma^2)_\theta$ must equal the two poles $\lambda=\pm \sqrt{n_2}$ of $\Sigma_{A(\Gamma^2)}(\lambda)$. By symmetry, we have $\sqrt{n_2}$ and $-\sqrt{n_2}$ as eigenvalues each with multiplicity $(m_1-n_1)$.
    \end{proof}
    \begin{cor}Consider $\gamma_1$-regular signed graph $\Gamma^1=(G,\sigma_1,\mu_1)$ with $n_1$ vertices and $m_1$ edges. Let $\theta$ be any $r$-orientation of edges of $\Gamma^1$. Suppose $\Gamma^2=(H,\sigma_2,\mu_2)$ is $(\gamma,k)$ co-regular graph with $n_2$ vertices and spectrum $(\beta_1,\beta_2,\cdots,\beta_{n_{2}})$, where the multiplicity of eigenvalue $k$ is $q$. If the spectrum of $\Gamma^1$ is $(\alpha_1,\alpha_2,\cdots,\alpha_{n_{1}})$, then the spectrum of $(\Gamma^1 \diamond \Gamma^2)_\theta$ consists of\\
    (i) $\beta_j$ each appearing with multiplicity $m_1$ corresponding to every eigenvalue $\beta_j$ of $\Gamma^2$ except for $\beta_j=k$.\\
    (ii) {\large$\frac{k+\alpha_j\pm \sqrt{(k-\alpha_j)^2+4n_2(\alpha_j+\gamma_1)}}{2}$} corresponding to each eigenvalue $\alpha_j$ of $\Gamma^1$.\\
    (iii) $k$ with multiplicity $m_1q-n_1$.
    \end{cor}
    \begin{cor} Let $\Gamma^1$ be $\gamma_1$-regular signed graph and $\theta_1$ and $\theta_2$ be any $r$-orientations of edges of $\Gamma^1$. If $\Gamma^2$ and $\Gamma^3$ are two A-cospectral signed graphs such that $\Sigma_{\Gamma^2}(\lambda)=\Sigma_{\Gamma^3}(\lambda)$ then $(\Gamma^1\diamond\Gamma^2)_{\theta_1}$ and $(\Gamma^1\diamond\Gamma^3)_{\theta_2}$ are A-cospectral.
    \end{cor}
     Now we will discuss about the spectrum of signed Laplacian of edge corona product of signed graphs under $r$-orientation.\\
     Consider $\gamma_1$-regular signed graph $\Gamma^1=(G,\sigma_1,\mu_1)$ on $n_1$ vertices and $m_1$ edges. Let $\theta$ be any $r$-orientation of edges of $\Gamma^1$ and $\Gamma^2=(H,\sigma_2,\mu_2)$ be arbitrary signed graph on $n_2$ vertices, $m_2$ edges then the signed Laplacian and signless Laplacian matrix of $(\Gamma^1 \diamond \Gamma^2)_\theta$ are respectively given by\\
    \begin{equation}
    \label{eqn 3.1}
 L(\Gamma^1 \diamond \Gamma^2)_\theta=
 \begin{bmatrix}
           L(\Gamma^1)+\gamma_1n_2I_{n_{_1}} &-R(\Gamma^1_\theta) \otimes \mu(\Gamma^2)^T \\ \\
            -R(\Gamma^1_\theta)^T \otimes \mu(\Gamma^2)  &I_{m_{_1}} \otimes(2I_{n_{_2}}+ L(\Gamma^2)) \\
         \end{bmatrix}
 \end{equation}
 and
 \begin{equation}
 \label{eqn 3.2}
     Q(\Gamma^1 \diamond \Gamma^2)_\theta=
         \begin{bmatrix}
                  Q(\Gamma^1)+\gamma_1n_2 I_{n_{_1}} &R(\Gamma^1_\theta) \otimes\mu(\Gamma^2)^T\\ \\
                  R(\Gamma^1_\theta)^T \otimes \mu(\Gamma^2) & I_{m_{_1}}\otimes \big(2 I_{n_{_2}}+Q(\Gamma^2)\big)\\
        \end{bmatrix}
 \end{equation}
\begin{thm}\label{thm 3.5} Consider $\gamma_1$-regular signed graph $\Gamma^1=(G,\sigma_1,\mu_1)$ on $n_1$ vertices, $m_1$ edges. Let $\theta$ be any $r$-orientation of edges of $\Gamma^1$ and $\Gamma_2=(H,\sigma_2,\mu_2)$ is any signed graph on $n_2$ vertices and $m_2$ edges. If $\lambda$ is not a pole of $\Sigma_{L(\Gamma^2)}(\lambda-2)$ then the characteristic polynomial of signed Laplacian matrix of $(\Gamma_1 \diamond \Gamma_2)_\theta$ is given by
 \begin{equation*}
 f_{L((\Gamma_1 \diamond \Gamma_2)_\theta)}(\lambda)=\big(f_{L(\Gamma^2)}(\lambda-2)\big)^{m_{_1}}\prod_{j=1}^{n_1} \Big[ \lambda - \lambda_j(L(\Gamma^1))-\gamma_1n_2+\big(\lambda_j(L(\Gamma^1))-2\gamma_1\big)\Sigma_{L(\Gamma^2)}(\lambda-2)\Big]
 \end{equation*}
 \end{thm}
\begin{proof} If we consider $R(\Gamma^1_\theta)$ as the vertex-edge incidence matrix of $\Gamma^1_\theta$, then using equation \ref{eqn 3.1} we have\\ 
\begin{equation*}
    \begin{split}
        f_{L((\Gamma_1 \diamond \Gamma_2)_\theta)}(\lambda)&=
       det\begin{bmatrix}
                  (\lambda-\gamma_1n_2) I_{n_{_1}}-L(\Gamma^1) &R(\Gamma^1_\theta) \otimes\mu(\Gamma^2)^T\\ \\
                  R(\Gamma^1_\theta)^T \otimes \mu(\Gamma^2) & I_{m_{_1}}\otimes \big((\lambda-2) I_{n_{_2}}-L(\Gamma^2)\big)
        \end{bmatrix}\\
        &=det\big(I_{m_{_1}}\otimes \big((\lambda-2) I_{n_{_2}}-L(\Gamma^2)\big)\big)~det\Big[(\lambda-\gamma_1n_2) I_{n_{_1}}-L(\Gamma^1)-\big(R(\Gamma^1_\theta)\otimes \mu(\Gamma^2)^T\big)\\
        &~~~~~~~~~~~~~~~~~~~~~~~~~~~~~~~~~~~~~~~~~~\big(I_{m_{_1}}\otimes\big((\lambda-2) I_{n_{_2}}-L(\Gamma^2)\big)\big)^{-1}\big(R(\Gamma^1_\theta)^T \otimes \mu(\Gamma^2)\big)\Big]\\
      &=det\big(I_{m_{_1}}\otimes \big((\lambda-2) I_{n_{_2}}-L(\Gamma^2)\big)\big)~det\Big[(\lambda-\gamma_1n_2) I_{n_{_1}}-L(\Gamma^1)-\big(R(\Gamma^1_\theta)I_{m_{_1}}R(\Gamma^1_\theta)^T\big)\\
      &~~~~~~~~~~~~~~~~~~~~~~~~~~~~~~~~~~~~~~~~~~~~~~~~~~~~\otimes \big(\mu(\Gamma^2)^T\big((\lambda-2) I_{n_{_2}}-L(\Gamma^2)\big)^{-1}\mu(\Gamma^2)\big)\Big]\\
      &=det\big(I_{m_{_1}}\otimes \big((\lambda-2) I_{n_{_2}}-L(\Gamma^2)\big)\big)~det\Big[(\lambda-\gamma_1n_2) I_{n_{_1}}-L(\Gamma^1)-Q(\Gamma^1) \otimes \Sigma_{L(\Gamma^2)}(\lambda-2)\Big]\\
      &=det\big(I_{m_{_1}}\otimes \big((\lambda-2) I_{n_{_2}}-L(\Gamma^2)\big)\big)~det\Big[(\lambda-\gamma_1n_2) I_{n_{_1}}-D(\Gamma^1)+A(\Gamma^1)\\
      &~~~~~~~~~~~~~~~~~~~~~~~~~~~~~~~~~~~~~~~~~~~~~~~~~~~~~~~~~~~~~~~~~~~~~~-(D(\Gamma^1)+A(\Gamma^1))~ \Sigma_{L(\Gamma^2)}(\lambda-2)\Big]\\
      &=det\big(I_{m_{_1}}\otimes \big((\lambda-2) I_{n_{_2}}-L(\Gamma^2)\big)\big)~det\Big[(\lambda-\gamma_1n_2) I_{n_{_1}}-\gamma_1(1+\Sigma_{L(\Gamma^2)}(\lambda-2))I_{n_{_1}}\\
      &~~~~~~~~~~~~~~~~~~~~~~~~~~~~~~~~~~~~~~~~~~~~~~~~~~~~~~~~~~~~~~~~~~~~~~~~~~~~~~- \big(\Sigma_{L(\Gamma^2)}(\lambda-2)-1\big)A(\Gamma^1)\Big]\\
        \end{split}
    \end{equation*}\\
    As $\Gamma^1$ is $\gamma_1$-regular, $\lambda_j(A(\Gamma^1))=\gamma_1-\lambda_j(L(\Gamma^1));j=1,2,\cdots,n_1$. Also using similar argument that we used in Theorem \ref{thm 3.1} we can say that $f_{L(\Gamma_\theta)}(\lambda)$ is independent of $\theta$ and is given by\\
    \begin{equation*}
        f_{L(\Gamma_\theta)}(\lambda)=\big(f_{L(\Gamma^2)}(\lambda-2)\big)^{m_{_1}}\prod_{j=1}^{n_1} \Big[ \lambda - \lambda_j(L(\Gamma^1))-\gamma_1n_2+\big(\lambda_j(L(\Gamma^1))-2\gamma_1\big)\Sigma_{L(\Gamma^2)}(\lambda-2)\Big] \qedhere
    \end{equation*}
    \end{proof}
    \begin{prop} Consider $\gamma_1$-regular signed graph $\Gamma^1=(G,\sigma_1,\mu_1)$ on $n_1$ vertices, $m_1$ edges. Let $\theta$ be any $r$-orientation of edges of $\Gamma^1$ and $\Gamma_2=(K_{1,n_2},\sigma_2,\mu_2)$ be a signed star with $V(\Gamma^2)=\{v_1,v_1,\cdots,v_{n_2+1}\}$ where $d(v_1)=n_2$ and $\mu_2=\mu^p$ or $\mu^c$. Suppose the spectrum of $L(\Gamma^1)$ is $(\alpha_1,\alpha_2,\cdots,\alpha_{n_{_1}})$, then the spectrum of $L(\Gamma^1 \diamond \Gamma^2)_\theta$ consists of\\
    (i) $3$ with multiplicity $m_1(n_2-1)$.\\
    (ii) The roots of the equation\\
    $[x-\alpha_j-\gamma_1(n_2+1)](x-2)(x-n_2-3)+(\alpha_j-2\gamma_1)[(n_2+1)(x-2)-(n_2+\mu_2(v_1))^2]=0$ corresponding to each eigenvalue $\alpha_j$ of $L(\Gamma^1)$.\\
    (iii) $2$ and $n_2+3$ each with multiplicity $m_1-n_1$.
    \end{prop}
    \begin{proof} The spectrum of $L(\Gamma^2)$ is $(0,n_2+1,1^{(n_2-1)})$. By Lemma \ref{lem 1.9}(ii)
    \[
    \Sigma_{L(\Gamma^2)}(\lambda-2)=\frac{(n_2+1)(\lambda-2)-(n_2^2+1)-2n_2\mu_2(v_1)}{(\lambda-2)(\lambda-3-n_2)}
    \]
    Two poles of $\Sigma_{L(\Gamma^2)}(\lambda-2)$ are $\lambda=2$ and $\lambda=n_2+3$. Suppose that $\lambda$ is not a pole of $\Sigma_{L(\Gamma^2)}(\lambda-2)$ then by Theorem \ref{thm 3.5}, the spectrum of $L(\Gamma^1 \diamond \Gamma^2)_\theta$ is given by
    \begin{itemize}
        \item $3$ with multiplicity $m_1(n_2-1)$.
        \item The roots of the equation
        \begin{equation*}
            \begin{split}
               &\lambda-\alpha_j-\gamma_1(n_2+1)+(\alpha_j-2\gamma_1)\Sigma_{L(\Gamma^2)}(\lambda-2)=0\\
               i.e.~&[\lambda-\alpha_j-\gamma_1(n_2+1)](\lambda-2)(\lambda-n_2-3)+(\alpha_j-2\gamma_1)[(n_2+1)(\lambda-2)-(n_2+\mu_2(v_1))^2]=0\\
                &\text{ corresponding to each eigenvalue } \alpha_j \text{ of } \Gamma^1.
            \end{split}
        \end{equation*}
    \end{itemize}
    The remaining $n_1+m_1(n_2+1)-(m_1(n_2-1)+3n_1)=2(m_1-n_1)$ eigenvalues of $L(\Gamma^1 \diamond \Gamma^2)_\theta$ must equal the two poles $\lambda=2$ and $\lambda=n_2+3$ of $\Sigma_{L(\Gamma^2)}(\lambda-2)$. By symmetry, we have $2$ and $n+3$ as eigenvalues of $L(\Gamma^1 \diamond \Gamma^2)_\theta$ each with multiplicity $m_1-n_1$.
    \end{proof}
    \begin{cor} Consider $\gamma_1$-regular signed graph $\Gamma^1=(G,\sigma_1,\mu_1)$ with $n_1$ vertices and $m_1$ edges. Let $\theta$ be any $r$-orientation of edges of $\Gamma^1$. Suppose $\Gamma^2=(H,\sigma_2,\mu_2)$ be $(\gamma,k)$ co-regular graph on $n_2$ vertices and Laplacian spectrum $(\beta_1,\beta_2,\cdots,\beta_{n_{2}})$, where the multiplicity of eigenvalue $\gamma-k$ of $L(\Gamma^2)$ is $q$. If the spectrum of $L(\Gamma^1)$ is $(\alpha_1,\alpha_2,\cdots,\alpha_{n_{1}})$ then the spectrum of $L(\Gamma^1 \diamond \Gamma^2)_\theta$ consists of\\
    (i) $\beta_j+2$ each appearing with multiplicity $m_1$ corresponding to every eigenvalue $\beta_j$ of $L(\Gamma^2)$ except when $\beta_j=\gamma-k$.\\
    (ii) $\frac{2+\gamma-k+\alpha_j+\gamma_1n_2\pm \sqrt{(2+\gamma-k-\alpha_j-\gamma_1n_2)^2+4(2\gamma_1-\alpha_j)}}{2}$ corresponding to each eigenvalue $\alpha_j$ of $L(\Gamma^1)$.\\
    (iii) $2+\gamma-k$ with multiplicity $m_1q-n_1$.
    \end{cor}
    \begin{cor} Consider $\gamma_1$-regular signed graph $\Gamma^1$. Let  $\theta_1$ and $\theta_2$ be any two $r$-orientations of edges of $\Gamma^1$. If $\Gamma^2$ and $\Gamma^3$ are two L-cospectral signed graphs such that $\Sigma_{L(\Gamma^2)}(\lambda)=\Sigma_{L(\Gamma^3)}(\lambda)$ then $(\Gamma^1\diamond\Gamma^2)_{\theta_1}$ and $(\Gamma^1\diamond\Gamma^3)_{\theta_2}$ are L-cospectral.
    \end{cor}
    \begin{thm}\label{thm 3.7} Consider $\gamma_1$-regular signed graph $\Gamma^1=(G,\sigma_1,\mu_1)$ on $n_1$ vertices, $m_1$ edges. Let $\theta$ be any $r$-orientation of edges of $\Gamma^1$ and $\Gamma^2=(H,\sigma_2,\mu_2)$ is any signed graph on $n_2$ vertices and $m_2$ edges. If $\lambda$ is not a pole of $\Sigma_{Q(\Gamma^2)}(\lambda-2)$ then the characteristic polynomial of signless Laplacian matrix of $(\Gamma^1 \diamond \Gamma^2)_\theta$ is given by
 \begin{equation*}
 f_{Q((\Gamma_1 \diamond \Gamma_2)_\theta)}(\lambda)=\big(f_{Q(\Gamma^2)}(\lambda-2)\big)^{m_{_1}}\prod_{j=1}^{n_1}\big[(\lambda-\gamma_1n_2)-\big(1+\Sigma_{Q(\Gamma^2)}(\lambda-2)\big)\lambda_j(Q(\Gamma^1))\big]
 \end{equation*}
 \end{thm}
  \begin{proof} If we consider $R(\Gamma^1_\theta)$ as the vertex-edge incidence matrix of $\Gamma^1_\theta$, then using equation \ref{eqn 3.2} we have\\
\begin{equation*}
    \begin{split}
        f_{Q((\Gamma_1 \diamond \Gamma_2)_\theta)}(\lambda)&=
       det\begin{bmatrix}
                  (\lambda-\gamma_1n_2) I_{n_{_1}}-Q(\Gamma^1) &-R(\Gamma^1_\theta) \otimes\mu(\Gamma^2)^T\\ \\
                  -R(\Gamma^1_\theta)^T \otimes \mu(\Gamma^2) & I_{m_{_1}}\otimes \big((\lambda-2) I_{n_{_2}}-Q(\Gamma^2)\big)
        \end{bmatrix}\\
       &=det\big(I_{m_{_1}}\otimes \big((\lambda-2) I_{n_{_2}}-Q(\Gamma^2)\big)\big)~det\Big[(\lambda-\gamma_1n_2) I_{n_{_1}}-Q(\Gamma^1)-\big(R(\Gamma^1_\theta)\otimes \mu(\Gamma^2)^T\big)\\
        &~~~~~~~~~~~~~~~~~~~~~~~~~~~~~~~~~~~~~~~~~~~~~~~~~~\big(I_{m_{_1}}\otimes\big((\lambda-2) I_{n_{_2}}-Q(\Gamma^2)\big)\big)^{-1}\big(R(\Gamma^1_\theta)^T \otimes \mu(\Gamma^2)\big)\Big]\\
      &=det\big(I_{m_{_1}}\otimes \big((\lambda-2) I_{n_{_2}}-Q(\Gamma^2)\big)\big)~det\Big[(\lambda-\gamma_1n_2) I_{n_{_1}}-Q(\Gamma^1)-\big(R(\Gamma^1_\theta)I_{m_{_1}}R(\Gamma^1_\theta)^T\big)\\
      &~~~~~~~~~~~~~~~~~~~~~~~~~~~~~~~~~~~~~~~~~~~~~~~~~~~~~~~~~~~~~~\otimes \big(\mu(\Gamma^2)^T\big((\lambda-2) I_{n_{_2}}-Q(\Gamma^2)\big)^{-1}\mu(\Gamma^2)\big)\Big]\\
      &=det\big(I_{m_{_1}}\otimes \big((\lambda-2) I_{n_{_2}}-Q(\Gamma^2)\big)\big)~det\Big[(\lambda-\gamma_1n_2) I_{n_{_1}}-Q(\Gamma^1)-Q(\Gamma^1) \otimes \Sigma_{Q(\Gamma^2)}(\lambda-2)\Big]\\
      &=det\big(I_{m_{_1}}\otimes \big((\lambda-2) I_{n_{_2}}-Q(\Gamma^2)\big)\big)~det\Big[(\lambda-\gamma_1n_2) I_{n_{_1}}- \big(1+\Sigma_{Q(\Gamma^2)}(\lambda-2)\big)Q(\Gamma^1)\Big]\\
      &=\big(f_{Q(\Gamma^2)}(\lambda-2)\big)^{m_{_1}}\prod_{j=1}^{n_1}\big[(\lambda-\gamma_1n_2)-\big(1+\Sigma_{Q(\Gamma^2)}(\lambda-2)\big)\lambda_j(Q(\Gamma^1))\big]
        \end{split}
    \end{equation*}
    We can easily see that the characteristic polynomial of signless Laplacian matrix of $(\Gamma^1 \diamond \Gamma^2)_\theta$ is independent of $\theta$.
    \end{proof}
     \begin{cor} Consider $\gamma_1$-regular signed graph $\Gamma^1=(G,\sigma_1,\mu_1)$ on $n_1$ vertices, $m_1$ edges. Let $\theta$ be any $r$-orientation of edges of $\Gamma^1$ and $\Gamma^2=(K_{1,n_2},\sigma_2,\mu_2)$ be a signed star with $V(\Gamma^2)=\{v_1,v_1,\cdots,v_{n_2+1}\}$ where $d(v_1)=n_2$. Suppose that the spectrum of $Q(\Gamma^1)$ is $(\alpha_1,\alpha_2,\cdots,\alpha_{n_{_1}})$, then the spectrum of $Q(\Gamma^1 \diamond \Gamma^2)_\theta$ consists of\\
    (i) $3$ with multiplicity $m_1(n_2-1)$.\\
    (ii) The roots of the equation $x^3-\big[n_2+5+\gamma_1(n_2+1)+\alpha_j\big]x^2+\big[\gamma_1(n_2+5)(n_2+1)+2(n_2+3)+4\alpha_j\big]x-\big[2\gamma_1(n_2+3)(n_2+1)-\alpha_j\big(n_2-\mu_2(v_1)\big)^2+4\alpha_j\big]=0$ corresponding to each eigenvalue $\alpha_j$ of $Q(\Gamma^1)$.\\
    (iii) $2$ and $n_2+3$ each with multiplicity $m_1-n_1$.
    \end{cor}
    \begin{cor} Consider $\gamma_1$-regular signed graph $\Gamma^1=(G,\sigma_1,\mu_1)$ on $n_1$ vertices and $m_1$ edges. Let $\theta$ be any $r$-orientation of edges of $\Gamma^1$. Suppose $\Gamma^2=(H,\sigma_2,\mu_2)$ be $(\gamma,k)$ co-regular signed graph on $n_2$ vertices, $m_2$ and signless Laplacian spectrum $(\beta_1,\beta_2,\cdots,\beta_{n_{2}})$ where multiplicity of eigenvalue $\gamma+k$ of $Q(\Gamma^2)$ is $q$. If the spectrum of $Q(\Gamma^1)$ is $(\alpha_1,\alpha_2,\cdots,\alpha_{n_{1}})$ then the spectrum of $Q(\Gamma^1 \diamond \Gamma^2)_\theta$ consists of\\
    (i) $\beta_j+2$ each appearing with multiplicity $m_1$ corresponding to every eigenvalue $\beta_j(\neq \gamma+k)$ of $Q(\Gamma^2)$.\\
    (ii) $\frac{2+\gamma+k+\gamma_1n_2+\alpha_j\pm \sqrt{(2+\gamma+k-\gamma_1n_2-\alpha_j)^2+4\alpha_jn_2}}{2}$ corresponding to each eigenvalue $\alpha_j$ of $Q(\Gamma^1)$.\\
    (iii) Eigenvalue $2+\gamma+k$ with multiplicity $m_1q-n_1$.
    \end{cor}
    \begin{cor} Consider $\gamma_1$-regular signed graph $\Gamma^1$. Let  $\theta_1$ and $\theta_2$ be any $r$-orientations of edges of $\Gamma^1$. If $\Gamma^2$ and $\Gamma^3$ are two Q-cospectral signed graphs such that $\Sigma_{Q(\Gamma^2)}(\lambda)=\Sigma_{Q(\Gamma^3)}(\lambda)$ then $(\Gamma^1\diamond\Gamma^2)_{\theta_1}$ and $(\Gamma^1\diamond\Gamma^3)_{\theta_2}$ are Q-cospectral.
    \end{cor}
 \section[Short Title for TOC and Headers]{Spectrum and Laplacian spectrum of $(\Gamma^1\boxdot \Gamma^2)_\theta$ and $(\Gamma^1\boxminus \Gamma^2)_\theta$}\label{subdivision}
     
     \begin{thm}\label{thm 4.1} Consider $\gamma_1$-regular graph $\Gamma^1=(G,\sigma_1,\mu_1)$ on $n_1$ vertices and $m_1$ edges. Let $\theta$ be any $r$-orientation of edges of $\Gamma^1$ and $\Gamma^2=(H,\sigma_2,\mu_2)$ be any arbitrary signed graph on $n_2$ vertices then
     \[
     f_{A(\Gamma^1 \boxdot \Gamma^2)_\theta}(\lambda)=\lambda^{m_1-n_1}\Big(f_{A(\Gamma^2)}(\lambda)\Big)^{n_1}~\prod_{j=1}^{n_1}\bigg(\lambda^2-\big(1+\lambda\Gamma_{A(\Gamma^2)}(\lambda)\big)\big(\lambda_j(\Gamma^1)+\gamma_1\big)\bigg)
     \]
     \end{thm}
     \begin{proof} If we consider $R(\Gamma^1_\theta)$ as the vertex-edge incidence matrix of $\Gamma^1_\theta$, then with respect to the partition \ref{eqn 1.4} the adjacency matrix of $(\Gamma^1 \boxdot \Gamma^2)_\theta$ is\\
     \begin{equation*}
         \begin{split}
             A(\Gamma^1 \boxdot \Gamma^2)_\theta&=\begin{bmatrix}
                 0_{n_1 \times n_1}&& R(\Gamma^1_\theta)&& 0_{n_1 \times n_1} \otimes \mu(\Gamma^2)^T\\ \\
                 R(\Gamma^1_\theta)&& 0_{m_1 \times m_1}&& R(\Gamma^1_\theta)^T \otimes \mu(\Gamma^2)^T\\ \\
                 0_{n_1 \times n_1}\otimes \mu(\Gamma^2)&& R(\Gamma^1_\theta)\otimes \mu(\Gamma^2)&& I_{n_1}\otimes A(\Gamma^2)
             \end{bmatrix}\\ \\
            \therefore f_{A(\Gamma^1 \boxdot \Gamma^2)_\theta}(\lambda)&=det\begin{bmatrix}
                 \lambda I_{n_1}&& -R(\Gamma^1_\theta)&& 0_{n_1 \times n_1} \otimes \mu(\Gamma^2)^T\\ \\
                 -R(\Gamma^1_\theta)&& \lambda I_{m_1}&& -R(\Gamma^1_\theta)^T \otimes \mu(\Gamma^2)^T\\ \\
                 0_{n_1 \times n_1}\otimes \mu(\Gamma^2)&& -R(\Gamma^1_\theta)\otimes \mu(\Gamma^2)&& I_{n_1}\otimes (\lambda I_{n_2}-A(\Gamma^2))
             \end{bmatrix}\\ \\
             &=det \Big(I_{n_1} \otimes \big(\lambda I_{n_2}-A(\Gamma^2)\big)\Big). det(S)\\ \\
         \end{split}
     \end{equation*}
     where $S=\begin{bmatrix}
         \lambda I_{n_1}&& -R(\Gamma^1_\theta)\\ \\
         -R(\Gamma^1_\theta)^T&& \lambda I_{m_1} - \Sigma_{A(\Gamma^2)}(\lambda)R(\Gamma^1_\theta)^T R(\Gamma^1_\theta)
     \end{bmatrix}$ is the schur complement \ref{schur} of $I_{n_1}\otimes\\ \\ \big(\lambda I_{n_2}-A(\Gamma^2)\big)$. Thus\\
     \begin{equation*}
         \begin{split}
             f_{A(\Gamma^1 \boxdot \Gamma^2)_\theta}(\lambda)&= \big(f_{A(\Gamma^2)}(\lambda)\big)^{n_1}.~det\big[\lambda I_{n_1}\big].~det\big[\lambda I_{m_1}-\Sigma_{A(\Gamma^2)}(\lambda)R(\Gamma^1_\theta)^TR(\Gamma^1_\theta)-\frac{1}{\lambda}R(\Gamma^1_\theta)^TR(\Gamma^1_\theta)\big]\\
             &=\lambda^{n_1}\big(f_{A(\Gamma^2)}(\lambda)\big)^{n_1}det\big[\lambda I_{m_1}-\big(\frac{1}{\lambda}+\Sigma_{A(\Gamma^2)}(\lambda)\big)R(\Gamma^1_\theta)^TR(\Gamma^1_\theta)\big]\\
             &=\lambda^{n_1}\big(f_{A(\Gamma^2)}(\lambda)\big)^{n_1} det\big[\lambda I_{m_1}-\big(\frac{1}{\lambda}+\Sigma_{A(\Gamma^2)}(\lambda)\big)\big(2I_{m_1}+A(\mathcal{L}(\Gamma^1_\theta))\big)\big]\\
              &=\lambda^{n_1}\big(f_{A(\Gamma^2)}(\lambda)\big)^{n_1}~\prod_{j=1}^{m_1}\bigg(\lambda-\big(\frac{1}{\lambda}+\Sigma_{A(\Gamma^2)}(\lambda)\big)\big(2+\lambda_j(\mathcal{L}(\Gamma^1_\theta))\big)\bigg)\\
             \end{split}
     \end{equation*}
     Using Remark \ref{rem 1.7} we can conclude that for any two $r$-orientation $\theta$ and $\theta'$ of $\Gamma^1$, $A(\mathcal{L}(\Gamma^1_\theta))$ and $A(\mathcal{L}(\Gamma^1_{\theta'}))$ are signature similar that is they have same set of eigenvalues. Also by Lemma \ref{lem 1.8} the eigenvalues of $\mathcal{L}(\Gamma^1_\theta)$ are $-2$ with multiplicity $m_1-n_1$ and $\lambda_j(\Gamma^1)+\gamma_1-2$; $j=1,2,\cdots,n_1$. This implies the characteristic polynomial of $(\Gamma^1 \boxdot \Gamma^2)_\theta$ is independent of $\theta$ and so we can simply write
     \begin{equation*}
         \begin{split}
             f_{\Gamma^1 \boxdot \Gamma^2}(\lambda)&=\lambda^{m_1}\big(f_{A(\Gamma^2)}(\lambda)\big)^{n_1}~\prod_{j=1}^{n_1}\bigg(\lambda-\big(\frac{1}{\lambda}+\Sigma_{A(\Gamma^2)}(\lambda)\big)\big(\gamma_1+\lambda_j(\Gamma^1)\big)\bigg)\\
             &=\lambda^{m_1-n_1}\big(f_{A(\Gamma^2)}(\lambda)\big)^{n_1}~\prod_{j=1}^{n_1}\bigg(\lambda^2-\big(1+\lambda \Sigma_{A(\Gamma^2)}(\lambda)\big)\big(\gamma_1+\lambda_j(\Gamma^1)\big)\bigg)
         \end{split}
     \end{equation*}
     Here $f_{\Gamma^1 \boxdot \Gamma^2}(\lambda)$ represent characteristic polynomial of $\Gamma^1 \boxdot \Gamma^2$ under any $r$-orientation of edges of $\Gamma^1$.
     \end{proof}
     \begin{prop} Consider $\gamma_1$-regular signed graph $\Gamma^1=(G,\Sigma_1,\mu_1)$ on $n_1$ vertices, $m_1$ edges. Let $\theta$ be any $r$-orientation of edges of $\Gamma^1$ and $\Gamma^2=(K_{1,n_2},\sigma_2,\mu_2)$ be a signed star on $(n_2+1)$ vertices with $V(\Gamma^2)=\{v_1,v_2,\cdots,v_{n_2+1}\}$ where $d(v_1)=n_2$ and $\mu_2=\mu^c$ or $\mu^p$. Suppose that the spectrum of $\Gamma^1$ is $(\alpha_1,\alpha_2,\cdots,\alpha_{n_1})$ then the spectrum of $\Gamma^1 \boxdot \Gamma^2$ consists of\\
     (i) $0$ with multiplicity $m_1+n_1(n-2)$.\\
     (ii) The roots of the equation
     \[
     x^4-\big(n_2+(\gamma_1+\alpha_j)(n_2+2)\big)x^2-2n_2(\gamma_1+\alpha_j)\mu_2(v_1)x+n_2(\gamma_1+\alpha_j)=0
     \]
     corresponding to each eigenvalue $\alpha_j$ of $\Gamma^1$.
     \end{prop}
     \begin{proof} The Spectrum of $\Gamma^2$ is $(0^{n_2-1},\pm\sqrt{n_2})$. By Lemma \ref{lem 1.9}(i)
     \[
     \Sigma_{A(\Gamma^2)}(\lambda)=\frac{(n_2+1)\lambda+2n_2\mu_2(v_1)}{\lambda^2-n_2}\\
     \]
     The two poles of $\Sigma_{A(\Gamma^2)}(\lambda)$ are $\lambda=\pm \sqrt{n_2}$. By Theorem \ref{thm 4.1}, the spectrum of $\Gamma^1 \boxdot \Gamma^2$ is given by
     \begin{itemize}
         \item The eigenvalue $0$ repeated $m_1-n_1+n_1(n_2-1)(=m_1+n_1(n_2-2))$ times.
         \item The roots of the equation obtained by solving
         \begin{equation*}
             \begin{split}
                 &\lambda^2-(1+\lambda \Sigma_{A(\Gamma^2)}(\lambda))(\gamma_1+\lambda_j(\Gamma^1))=0\\
                 i.e~~&\lambda^4-\big(n_2+(\gamma_1+\alpha_j)(n_2+2)\big)\lambda^2-2n_2(\gamma_1+\alpha_j)\mu_2(v_1)\lambda+n_2(\gamma_1+\alpha_j)=0\\
                 &\text{ corresponding to each eigenvalue } \alpha_j(j=1,2,\cdots,n_1) \text{ of } \Gamma^1. \qedhere
             \end{split}
         \end{equation*}
     \end{itemize}
     \end{proof}
     \begin{cor} Consider $\gamma_1$-regular signed graph $\Gamma^1=(G,\sigma_1,\mu_1)$ on $n_1$ vertices and $m_1$ edges. Let $\theta$ be any $r$-orientation of edges of $\Gamma$. Suppose $\Gamma^2=(H,\sigma_2,\mu_2)$ be $(\gamma,k)$ co-regular signed graph on $n_2$ vertices and spectrum $(\beta_1,\beta_2,\cdots,\beta_{n_2})$ where multiplicity of eigenvalue $k$ is $q$. If the spectrum of $\Gamma^1$ is $(\alpha_1,\alpha_2,\cdots,\alpha_{n_1})$ then the spectrum of $\Gamma^1 \boxdot \Gamma^2$ consists of\\
     (i) Eigenvalue $0$ with multiplicity $m_1-n_1$.\\
     (ii) Eigenvalue $\beta_j$ each appearing with multiplicity $n_1$ corresponding to every eigenvalue $\beta_j$ of $\Gamma^2$ except for $\beta_j=k$.\\
     (iii) The roots of the equation
     $x^3-kx^2-(\gamma_1+\alpha_j)(1+n_2)x+k(\gamma_1+\alpha_j)=0$ corresponding to every eigenvalue $\alpha_j$ of $\Gamma^1$.\\
     (iv) Eigenvalue $k$ with multiplicity $n_1(q-1)$.
     \end{cor}
\begin{cor} Consider $\gamma_1$-regular signed graph $\Gamma^1$. Let $\theta_1$ and $\theta_2$ be any two $r$-orientations of edges of $\Gamma^1$. If $\Gamma^2$ and $\Gamma^3$ are two A-cospectral signed graphs such that $\Sigma_{A(\Gamma^2)}(\lambda)=\Sigma_{A(\Gamma^3)}(\lambda)$ then $(\Gamma^1\boxdot\Gamma^2)_{\theta_1}$ and $(\Gamma^1\boxdot\Gamma^3)_{\theta_2}$ are A-cospectral.
\end{cor}
\begin{thm}\label{thm 4.5} Consider $\gamma_1$-regular signed graph $\Gamma^1=(G,\sigma_1,\mu_1)$ on $n_1$ vertices and $m_1$ edges. Let $\theta$ be any $r$-orientation of edges  of $\Gamma^1$ and $\Gamma^2=(H,\sigma_2,\mu_2)$ be any arbitrary signed graph on $n_2$ vertices. If $\lambda$ is not a pole of $\Sigma_{L(\Gamma^2)}(\lambda-\gamma_1)$ then the characteristic polynomial of signed Laplacian matrix of $\Gamma^1 \boxdot \Gamma^2$ under the $r$-orientation $\theta$ is\\
\begin{equation*}
    \begin{split}
        f_{L(\Gamma^1 \boxdot \Gamma^2)_\theta}(\lambda)=&(\lambda-2-2n_2)^{m_1-n_1}\big(f_{L(\Gamma^2)}(\lambda-\gamma_1)\big)^{n_1} \prod_{j=1}^{n_1}\big[\lambda^2-(\gamma_1+2+2n_2)\lambda+2\gamma_1(n_2+1)\\
        &~~~~~~~~~~~~~~~~~~~~~~~~~~~~~~~~~~~~~~-\big(1+(\lambda-\gamma_1)\Sigma_{L(\Gamma^2)}(\lambda-\gamma_1)\big)(2\gamma_1-\lambda_j(L(\Gamma^1)))\big]\\
    \end{split}
\end{equation*}
\end{thm}
\begin{proof} If we consider $R(\Gamma^1_\theta)$ as the vertex-edge incidence matrix of $\Gamma^1_\theta$, then with respect to the partition \ref{eqn 1.4} the signed Laplacian matrix of $(\Gamma^1 \boxdot \Gamma^2)_\theta$ is given by
\begin{equation*}
         \begin{split}
             L(\Gamma^1 \boxdot \Gamma^2)_\theta&=\begin{bmatrix}
                 \gamma_1I_{n_1}&& -R(\Gamma^1_\theta)&& 0_{n_1 \times n_1} \otimes \mu(\Gamma^2)^T\\ \\
                 -R(\Gamma^1_\theta)&& (2+2n_2)I_{m_1}&& -R(\Gamma^1_\theta)^T \otimes \mu(\Gamma^2)^T\\ \\
                 0_{n_1 \times n_1}\otimes \mu(\Gamma^2)&& -R(\Gamma^1_\theta)\otimes \mu(\Gamma^2)&& I_{n_1}\otimes (\gamma_1I_{n_2}+L(\Gamma^2))
             \end{bmatrix}
             \end{split}
             \end{equation*}
             \begin{equation*}
             \begin{split}
            \therefore f_{L(\Gamma^1 \boxdot \Gamma^2)_\theta}(\lambda)&=det\begin{bmatrix}
                 (\lambda-\gamma_1) I_{n_1}&& R(\Gamma^1_\theta)&& 0_{n_1 \times n_1} \otimes \mu(\Gamma^2)^T\\ \\
                 R(\Gamma^1_\theta)^T&& (\lambda-2-2n_2) I_{m_1}&& R(\Gamma^1_\theta)^T \otimes \mu(\Gamma^2)^T\\ \\
                 0_{n_1 \times n_1}\otimes \mu(\Gamma^2)&& R(\Gamma^1_\theta)\otimes \mu(\Gamma^2)&& I_{n_1}\otimes ((\lambda-\gamma_1) I_{n_2}-L(\Gamma^2))
             \end{bmatrix}\\ \\
             &=det \Big(I_{n_1} \otimes \big((\lambda-\gamma_1) I_{n_2}-L(\Gamma^2)\big)\Big). det(S)
         \end{split}
     \end{equation*}
     where $S=\begin{bmatrix}
         (\lambda-\gamma_1) I_{n_1}&& R(\Gamma^1_\theta)\\ \\
         R(\Gamma^1_\theta)^T&& (\lambda-2-2n_2) I_{m_1} - \Sigma_{L(\Gamma^2)}(\lambda-\gamma_1)R(\Gamma^1_\theta)^T R(\Gamma^1_\theta)
     \end{bmatrix}$ is the schur complement \ref{schur} of $I_{n_1}\otimes \big((\lambda-\gamma_1)I_{n_2}-L(\Gamma^2)\big)$. Thus
     \begin{equation*}
         \begin{split}
             f_{L(\Gamma^1 \boxdot \Gamma^2)_\theta}(\lambda)&= \big(f_{L(\Gamma^2)}(\lambda-\gamma_1)\big)^{n_1}.~det\big[(\lambda-\gamma_1) I_{n_1}\big].~det\big[(\lambda-2-2n_2) I_{m_1}\\
             &~~~~~~~~~~~~~~~~~~~~~~~~~~~~~~~~~~~~~-\Sigma_{L(\Gamma^2)}(\lambda-\gamma_1)R(\Gamma^1_\theta)^TR(\Gamma^1_\theta)-\frac{1}{\lambda-\gamma_1}R(\Gamma^1_\theta)^TR(\Gamma^1_\theta)\big]\\
             &=(\lambda-\gamma_1)^{n_1}\big(f_{L(\Gamma^2)}(\lambda-\gamma_1)\big)^{n_1}det\bigg[(\lambda-2-2n_2) I_{m_1}\\
             &~~~~~~~~~~~~~~~~~~~~~~~~~~~~~~~~~~~~~~~~~-\bigg(\frac{1}{\lambda-\gamma_1}+\Sigma_{L(\Gamma^2)}(\lambda-\gamma_1)\bigg)\big(2I_{m_1}+A(\mathcal{L}(\Gamma_\theta^1))\big)\bigg]\\
             &=(\lambda-\gamma_1)^{n_1}\big(f_{L(\Gamma^2)}(\lambda-\gamma_1)\big)^{n_1} \prod_{j=1}^{m_1}\bigg[(\lambda-2-2n_2)-\bigg(\frac{1}{\lambda-\gamma_1}+\Sigma_{A(\Gamma^2)}(\lambda)\bigg)\\      &~~~~~~~~~~~~~~~~~~~~~~~~~~~~~~~~~~~~~~~~~~~~~~~~~~~~~~~~~~~~~~~~~~~~~~~~~~~~~~~~~~~~~\big(2+\lambda_j(\mathcal{L}(\Gamma^1_\theta))\big)\bigg]\\
             \end{split}
     \end{equation*}
As $\Gamma^1$ is $\gamma_1$-regular, $\lambda_j(L(\Gamma^1))=\gamma_1-\lambda_j(\Gamma^1)$ for $j=1,2,\cdots,n_1$. Also using Remark \ref{rem 1.7}, Lemma \ref{lem 1.8} and similar argument that we used in Theorem \ref{thm 4.1} we can say that the characteristic polynomial of $L(\Gamma^1 \boxdot  \Gamma^2)_\theta$ is independent of $\theta$ and is given by     
     \begin{equation*}
         \begin{split}
             f_{L(\Gamma^1 \boxdot \Gamma^2)}(\lambda)&=(\lambda-2-2n_2)^{m_1-n_1}(\lambda-\gamma_1)^{n_1}\big(f_{L(\Gamma^2)}(\lambda-\gamma_1)\big)^{n_1}~\prod_{j=1}^{n_1}\bigg[(\lambda-2-2n_2)\\
             &~~~~~~~~~~~~~~~~~~~~~~~~~~~~~~~~~~~~~~~~~~~~~~~~~~~~~-\big(\frac{1}{\lambda-\gamma_1}+\Sigma_{L(\Gamma^2)}(\lambda-\gamma_1)\big)\big(\gamma_1+\lambda_j(\Gamma^1)\big)\bigg]\\
             &=(\lambda-2-2n_2)^{m_1-n_1}\big(f_{L(\Gamma^2)}(\lambda-\gamma_1)\big)^{n_1}~\prod_{j=1}^{n_1}\bigg[\lambda^2-(\gamma_1+2+2n_2)\lambda+2\gamma_1(n_2+1)\\
             &~~~~~~~~~~~~~~~~~~~~~~~~~~~~~~~~~~~~~~~~~~~~-\big(1+(\lambda-\gamma_1) \Sigma_{L(\Gamma^2)}(\lambda-\gamma_1)\big)\big(2\gamma_1-\lambda_j(L(\Gamma^1))\big)\bigg] \qedhere
         \end{split}
     \end{equation*}
     \end{proof}
\begin{prop} Consider $\gamma_1$-regular signed graph $\Gamma^1=(G,\sigma_1,\mu_1)$ on $n_1$ vertices and $m_1$ edges. Let $\theta$ be any $r$-orientation of edges of $\Gamma^1$ and $\Gamma^2=(K_{1,n_2},\sigma_2,\mu_2)$ be a signed star with $V(\Gamma^2)=\{v_1,v_1,\cdots,v_{n_2+1}\}$ where $d(v_1)=n_2$ and $\mu_2=\mu^p$ or $\mu^c$. Suppose the spectrum of $L(\Gamma^1)$ is $(\alpha_1,\alpha_2,\cdots,\alpha_{n_{_1}})$, then the spectrum of $(\Gamma^1 \boxdot \Gamma_2)_\theta$ consists of\\
(i) Eigenvalues $2n_2+4$ with multiplicity $m_1-n_1$.\\
(ii) Eigenvalue $1+\gamma_1$ with multiplicity $n_1(n_2-1)$.\\
(iii) Eigenvalue $\gamma_1$ with multiplicity $n_1$.\\
(iv) The roots of the equation $x^3-(2\gamma_1+3n_2+5)x^2+\big[(\gamma_1+n_2+1)(\gamma_1+2n_2+4)+\alpha_j(n_2+2)\big]x+(n_2+\mu_2(v_1))^2(2\gamma_1-\alpha_j)-2\gamma_1(n_2+1)^2-\alpha_j\big(\gamma_1+(\gamma_1+1)(n_2+1)\big)=0$ corresponding to each eigenvalue $\alpha_j(j=1,2,\cdots,n_1)$ of $L(\Gamma^1)$.
\end{prop}
\begin{proof} The spectrum of $L(\Gamma^2)$ is $(0,1^{(n_2-1)},n_2+1)$. By Lemma \ref{lem 1.9}(ii)
\[
\Sigma_{L(\Gamma^2)}(\lambda-\gamma_1)=\frac{(n_2+1)(\lambda-\gamma_1)-(n_2^2+1)-2n_2\mu_2(v_1)}{(\lambda-\gamma_1)(\lambda-\gamma_1-(n_2+1))}
\]
Two poles of $\Sigma_{L(\Gamma^2)}(\lambda-\gamma_1)$ are $\lambda=\gamma_1$ and $\lambda=\gamma_1+n_2+1$. By Theorem \ref{thm 4.5}, the spectrum of $L(\Gamma^1\boxdot \Gamma^2)$ is given by
\begin{itemize}
    \item Eigenvalues $2n_2+4$ with multiplicity $m_1-n_1$.
    \item Eigenvalue $1+\gamma_1$ with multiplicity $n_1(n_2-1)$.
    \item The roots of the equation
    \begin{equation*}
             \begin{split}
             &\lambda^2-(\gamma_1+2n+4)\lambda-2\gamma_1(n+2)-\Big(1+(\lambda-\gamma_1)\Sigma_{L(\Gamma^2)}(\lambda-\gamma_1)\Big)(2\gamma_1-\alpha_j)=0\\
             i.e~~&\lambda=\gamma_1 \text{ with multiplicity $n_1$ and roots of the equation }\\
             &\lambda^3-(2\gamma_1+3n+5)\lambda^2+\big[(\gamma_1+n+1)(\gamma_1+2n+4)+\alpha_j(n+2)\big]\lambda\\
             &~~~~~~~~~~~~+(n+\mu_2(v_1))^2(2\gamma_1-\alpha_j)-2\gamma_1(n+1)^2-\alpha_j\big(\gamma_1+(\gamma_1+1)(n+1)\big)=0\\
             \end{split}
         \end{equation*}
         \end{itemize}
         corresponding to each eigenvalue $\alpha_j(j=1,2,\cdots,n_1)$ of $L(\Gamma^1)$.
\end{proof}
\begin{cor} Consider $\gamma_1$-regular signed graph $\Gamma^1=(G,\sigma_1,\mu_1)$ be on $n_1$ vertices and $m_1$ edges. Let $\theta$ be any $r$-orientation of edges of $\Gamma$. Suppose $\Gamma^2=(H,\sigma_2,\mu)$ be $(\gamma,k)$ co-regular signed graph on $n_2$ vertices and Laplacian spectrum $(\beta_1,\beta_2,\cdots,\beta_{n_2})$ where multiplicity of eigenvalue $r-k$ is $q$. If the Laplacian spectrum of $L(\Gamma^1)$ is  $(\alpha_1,\alpha_2,\cdots,\alpha_{n_1})$ then the spectrum of $L(\Gamma^1 \boxdot \Gamma^2)$ consists of\\
     (i) Eigenvalue $2+2n_2$ with multiplicity $m_1-n_1$.\\
     (ii) Eigenvalue $\beta_j+\gamma_1$ each appearing with multiplicity $n_1$ corresponding to every eigenvalue $\beta_j(\neq r-k)$ of $L(\Gamma^2)$.\\
     (iii) The roots of the equation
     $x^3-(2\gamma_1+\gamma-k+2n_2+2)x^2-\big[(\gamma_1+2+2n_2)(\gamma_1+\gamma-k)+\alpha_j(1+n_2)\big]x+\gamma_1(n_2+1)(2k-2\gamma-\alpha_j)+(2\gamma_1-\alpha_j)(\gamma-k)=0$ corresponding to each eigenvalue $\alpha_j$ of $L(\Gamma^1)$.\\
     (iv) Eigenvalue $\gamma_1+\gamma-k$ with multiplicity $n_1(q-1)$.
     \end{cor}
\begin{cor} Consider $\gamma_1$-regular signed graph $\Gamma^1$. Let $\theta_1$ and $\theta_2$ be any $r$-orientations of edges of $\Gamma^1$. If $\Gamma^2$ and $\Gamma^3$ are two L-cospectral signed graphs such that $\Sigma_{L(\Gamma^2)}(\lambda)=\Sigma_{L(\Gamma^3)}(\lambda)$ then $(\Gamma^1\boxdot\Gamma^2)_{\theta_1}$ and $(\Gamma^1\boxdot\Gamma^3)_{\theta_2}$ are L-cospectral.
\end{cor}
\begin{thm}\label{thm 4.9} Consider $\gamma_1$-regular signed graph $\Gamma^1=(G,\sigma_1,\mu_1)$ on $n_1$ vertices and $m_1$ edges. Let $\theta$ be any $r$-orientation of edges of $\Gamma^1$ and $\Gamma^2=(H,\sigma_2,\mu_2)$ be any arbitrary signed graph on $n_2$ vertices. If $\lambda$ is not a pole of $\Sigma_{L(\Gamma^2)}(\lambda-\gamma_1)$ then the characteristic polynomial of signless Laplacian matrix of $\Gamma^1 \boxdot \Gamma^2$ under the $r$-orientation $\theta$ is
\begin{equation*}
    \begin{split}
        f_{Q(\Gamma^1 \boxdot \Gamma^2)_\theta}(\lambda)=(\lambda-2-2n_2)^{m_1-n_1}\Big(f_{Q(\Gamma^2)}(\lambda-\gamma_1)\Big)^{n_1}~\prod_{i=1}^{n_1}\bigg(\lambda^2-(\gamma_1+2+2n_2)\lambda+2\gamma_1(n_2+1)\\
        ~~~~~~~~-\Big(1+(\lambda-\gamma_1)\Sigma_{Q(\Gamma^2)}(\lambda-\gamma_1)\Big)\lambda_j\big(Q(\Gamma^1)\big)\bigg)
    \end{split}
\end{equation*}
\end{thm}
\begin{proof} If we consider $R(\Gamma^1_\theta)$ as the vertex-edge incidence matrix of $\Gamma^1_\theta$, then with respect to the partition \ref{eqn 1.4} the signless Laplacian matrix of $(\Gamma^1 \boxdot \Gamma^2)_\theta$ is
\begin{equation*}
    Q(\Gamma^1 \boxdot \Gamma^2)_\theta=\begin{bmatrix}
        \gamma_1I_{n_1}&& R(\Gamma^1_\theta)&& 0_{n_1 \times n_1} \otimes \mu(\Gamma^2)^T\\ \\
        R(\Gamma^1_\theta)^T&& (2+2n_2)I_{m_1}&& R(\Gamma^1_\theta)^T \otimes \mu(\Gamma^2)^T\\ \\
        0_{n_1 \times n_1} \otimes \mu(\Gamma^2)&& R(\Gamma^1_\theta) \otimes \mu(\Gamma^2)&& I_{n_1} \otimes \big(\gamma_1I_{n_2}+Q(\Gamma^2)\big)
    \end{bmatrix}
\end{equation*}
The remainder of the proof follows from Theorem \ref{thm 4.5}.
\end{proof}
\begin{cor} Consider $\gamma_1$-regular signed graph $\Gamma^1=(G,\sigma_1,\mu_1)$ on $n_1$ vertices and $m_1$ edges. Let $\theta$ be any $r$-orientation of edges of $\Gamma^1$ and $\Gamma^2=(K_{1,n_2},\sigma_2,\mu_2)$ be a signed star with $V(\Gamma^2)=\{v_1,v_1,\cdots,v_{n_2+1}\}$ where $d(v_1)=n_2$ and $\mu_2=\mu^p$ or $\mu^c$. Suppose the spectrum of $Q(\Gamma^1)$ is $(\alpha_1,\alpha_2,\cdots,\alpha_{n_{1}})$, then the spectrum of $Q(\Gamma^1 \boxdot \Gamma_2)_\theta$ consists of\\
(i) Eigenvalue $2n_2+4$ with multiplicity $m_1-n_1$.\\
(ii) Eigenvalue $1+\gamma_1$ with multiplicity $n_1(n_2-1)$.\\
(iii) Eigenvalue $\gamma_1$ with multiplicity $n_1$.\\
(iv) The roots of the equation $x^3-(2\gamma_1 + 3n + 5)x^2 + [(\gamma_1 + n +
1)(\gamma_1+2n+4)+(2\gamma_1-\alpha_j)(n+2)]x+\alpha_j
[\gamma_1+(n+1)(\gamma_1+1)+(n-\mu_2(v_1))^2]-2\gamma_1(n+2)(\gamma_1+n+1) = 0$
corresponding to each eigenvalue $\alpha_j$ of $Q(\Gamma^1)$.
\end{cor}
\begin{cor} Consider $\gamma_1$-regular signed graph $\Gamma_1 = (G, \sigma_1, \mu_1)$ be on $n_1$ vertices and $m_1$ edges. Let $\theta$ be any $r$-orientation of edges of $\Gamma^1$. Suppose $\Gamma^2=(H,\sigma_2,\mu_2)$ be $(\gamma, k)$ co-regular signed graph on $n_2$ vertices and signless Laplacian spectrum $(\beta_1, \beta_2, \cdots , \beta_{n_2} )$ where multiplicity of eigenvalue $\gamma+k$ of $Q(\Gamma^2)$ is $q$. If the spectrum of $Q(\Gamma^1)$ is $(\alpha_1, \alpha_2,\cdots , \alpha_{n_1} )$ then the spectrum of $Q(\Gamma^1 \boxdot \Gamma^2)$ consists of\\
(i) Eigenvalue $2 + 2n_2$ with multiplicity $m_1 - n_1$.\\
(ii) Eigenvalue $\beta_j+\gamma_1$ each appearing with multiplicity $n_1$ corresponding to every eigenvalue $\beta_j(\neq \gamma+k)$ of $Q(\Gamma^2)$.\\
(iii) The roots of the equation $x^3 - (2\gamma_1 + \gamma + k + 2n_2 + 2)x^2 + [(\gamma_1 + 2 +
2n_2)(\gamma_1 + \gamma + k) + (2\gamma_1 - \alpha_j)(1 + n_2)]x + \gamma_1(n_2 + 1)(\alpha_j - 2\gamma_1 - 2\gamma - 2k) + (\gamma + k)\alpha_j = 0$ corresponding to every eigenvalue $\alpha_j(j = 1, 2, \cdots , n_1)$ of $Q(\Gamma^1)$.\\
(iv) Eigenvalue $\gamma_1 + \gamma + k$ with multiplicity $n_1(q - 1)$.
\end{cor}
\begin{cor} Let $\Gamma^1$ be $\gamma_1$-regular signed graph and $\theta_1$ and $\theta_2$ be any $r$-orientations of edges of $\Gamma^1$. If $\Gamma^2$ and $\Gamma^3$ are two Q-cospectral signed graphs such that $\Sigma_{Q(\Gamma^2)}(\lambda)=\Sigma_{Q(\Gamma^3)}(\lambda)$ then $(\Gamma^1\boxdot\Gamma^2)_{\theta_1}$ and $(\Gamma^1\boxdot\Gamma^3)_{\theta_2}$ are Q-cospectral.
\end{cor}
     Now we shall discuss about the characteristic polynomial for subdivision edge neighbourhood corona of signed graphs under $r$-orientation.
     \begin{thm}\label{thm 4.13} Consider $\gamma_1$-regular signed graph $\Gamma^1=(G,\sigma_1,\mu_1)$ on $n_1$ vertices and $m_1$ edges. Let $\theta$ be any $r$-orientation of $\Gamma^1$ and $\Gamma^2=(H,\sigma_2,\mu_2)$ be any arbitrary signed graph on $n_2$ vertices then
     \[
     f_{A(\Gamma^1 \boxminus \Gamma^2)_\theta}(\lambda)=\lambda^{m_1-n_1}\Big(f_{A(\Gamma^2)}(\lambda)\Big)^{m_1}~\prod_{j=1}^{n_1}\bigg(\lambda^2-\big(1+\lambda\Sigma_{A(\Gamma^2)}(\lambda)\big)\big(\lambda_j(\Gamma^1)+\gamma_1\big)\bigg)
     \]
     \end{thm}
     \begin{proof} If we consider $R(\Gamma^1_\theta)$ as the vertex-edge incidence matrix of $\Gamma^1_\theta$, then with respect to the partition \ref{eqn 1.5} the adjacency matrix of $(\Gamma^1 \boxminus \Gamma^2)_\theta$ is\\
     \begin{equation*}
             A(\Gamma^1 \boxminus \Gamma^2)_\theta=\begin{bmatrix}
                 0_{n_1 \times n_1}&& R(\Gamma^1_\theta)&& R(\Gamma^1_\theta) \otimes \mu(\Gamma^2)^T\\ \\
                 R(\Gamma^1_\theta)^T&& 0_{m_1 \times m_1}&& 0_{m_1\times m_1} \otimes \mu(\Gamma^2)^T\\ \\
                 R(\Gamma_\theta^1)^T\otimes \mu(\Gamma^2)&& 0_{m_1\times m_1} \otimes \mu(\Gamma^2)&& I_{m_1}\otimes A(\Gamma^2)
             \end{bmatrix}
             \end{equation*}
             \begin{equation*}
             \begin{split}
            \therefore f_{A(\Gamma^1 \boxminus \Gamma^2)_\theta} (\lambda)=det&\begin{bmatrix}
                 \lambda I_{n_1}&& -R(\Gamma^1_\theta)&& -R(\Gamma^1_\theta) \otimes \mu(\Gamma^2)^T\\ \\
                 -R(\Gamma^1_\theta)^T&& \lambda I_{m_1}&& 0_{m_1\times m_1} \otimes \mu(\Gamma^2)^T\\ \\
                 -R(\Gamma_\theta^1)^T\otimes \mu(\Gamma^2)&& 0_{m_1\times m_1} \otimes \mu(\Gamma^2)&& I_{m_1}\otimes (\lambda I_{n_2}A(\Gamma^2))
    \end{bmatrix}\\
             =det& \Big(I_{m_1} \otimes \big(\lambda I_{n_2}-A(\Gamma^2)\big)\Big). det(S)\\ \\
         \end{split}
     \end{equation*}
     where $S=\begin{bmatrix}
         \lambda I_{n_1} - \Sigma_{A(\Gamma^2)}(\lambda)R(\Gamma^1_\theta) R(\Gamma^1_\theta)^T&& -R(\Gamma^1_\theta)\\ \\
         -R(\Gamma^1_\theta)^T&& \lambda I_{m_1}
     \end{bmatrix}$ is the schur complement \ref{schur} of $I_{m_1}\otimes\\ \\ \big(\lambda I_{n_2}-A(\Gamma^2)\big)$. Thus
     \begin{equation*}
         \begin{split}
             f_{A(\Gamma^1 \boxminus \Gamma^2)_\theta}(\lambda)&= \big(f_{A(\Gamma^2)}(\lambda)\big)^{m_1}.~det\big[\lambda I_{m_1}\big].~det\big[\lambda I_{n_1}-\Sigma_{A(\Gamma^2)}(\lambda)R(\Gamma^1_\theta)R(\Gamma^1_\theta)^T-\frac{1}{\lambda}R(\Gamma^1_\theta)R(\Gamma^1_\theta)^T\big]\\
             &=\lambda^{m_1}\big(f_{A(\Gamma^2)}(\lambda)\big)^{m_1}det\big[\lambda I_{n_1}-\big(\frac{1}{\lambda}+\Sigma_{A(\Gamma^2)}(\lambda)\big)R(\Gamma^1_\theta)R(\Gamma^1_\theta)^T\big]\\
             &=\lambda^{m_1}\big(f_{A(\Gamma^2)}(\lambda)\big)^{m_1} det\big[\lambda I_{n_1}-\big(\frac{1}{\lambda}+\Sigma_{A(\Gamma^2)}(\lambda)\big)Q(\Gamma^1)\big]
             \end{split}
     \end{equation*}
     By Lemma \ref{lem 1.3} we have $R(\Gamma^1_\theta)R(\Gamma^1_\theta)^T=Q(\Gamma^1)$ for any $r$-orientation $\theta$ of $\Gamma^1$. Also as $\Gamma^1$ is $\gamma_1$-regular, $\lambda_j(Q(\Gamma^1))=\gamma_1+\lambda_j(\Gamma^1)$ $j=1,2,\cdots,n_1$. This implies the characteristic polynomial of $(\Gamma^1 \boxminus \Gamma^2)_\theta$ is independent of $\theta$ and so we can simply write
     \begin{equation*}
         \begin{split}
             f_{A(\Gamma^1 \boxminus \Gamma^2)}(\lambda)&=\lambda^{m_1}\big(f_{A(\Gamma^2)}(\lambda)\big)^{m_1}~\prod_{j=1}^{n_1}\bigg(\lambda-\big(\frac{1}{\lambda}+\Sigma_{A(\Gamma^2)}(\lambda)\big)\big(\gamma_1+\lambda_j(\Gamma^1)\big)\bigg)\\
             &=\lambda^{m_1-n_1}\big(f_{A(\Gamma^2)}(\lambda)\big)^{m_1}~\prod_{j=1}^{n_1}\bigg(\lambda^2-\big(1+\lambda \Sigma_{A(\Gamma^2)}(\lambda)\big)\big(\gamma_1+\lambda_j(\Gamma^1)\big)\bigg)
         \end{split}
     \end{equation*}
     Here $f_{A(\Gamma^1 \boxdot \Gamma^2)}(\lambda)$ represent characteristic polynomial of adjacency matrix of $\Gamma^1 \boxminus \Gamma^2$ under any $r$-orientation of edges of $\Gamma^1$.
     \end{proof}
     \begin{prop} Consider $\gamma_1$-regular signed graph $\Gamma^1=(G,\sigma_1,\mu_1)$ on $n_1$ vertices, $m_1$ edges. Let $\theta$ be any $r$-orientation of edges of $\Gamma^1$ and $\Gamma^2=(K_{1,n_2},\sigma_2,\mu_2)$ be a signed star with $V(\Gamma^2)=\{v_1,v_1,\cdots,v_{n_2+1}\}$ where $d(v_1)=n_2$ and $\mu_2=\mu^p$ or $\mu^c$. Suppose the spectrum of $\Gamma^1$ is $(\alpha_1,\alpha_2,\cdots,\alpha_{n_{1}})$, then the spectrum of $(\Gamma^1 \boxminus \Gamma_2)_\theta$ consists of\\
     (i) Eigenvalue 0 with multiplicity $m_1n_2-n_1$.\\
     (ii) The roots of the equation
     \[
     x^4-\big(n_2+(\gamma_1+\alpha_j)(n_2+2)\big)x^2-2n_2(\gamma_1+\alpha_j)\mu_2(v_1)x+n_2(\gamma_1+\alpha_j)=0
     \]
     corresponding to each eigenvalue $\alpha_j$ of $\Gamma^1$.\\
     (iii) Eigenvalues $\sqrt{n_2}$ and $-\sqrt{n_2}$ each with multiplicity $m_1-n_1$.
     \end{prop}
     \begin{proof} Spectrum of $\Gamma^2$ is $(-\sqrt{n_2},\sqrt{n_2},0^{n-1})$. By Lemma \ref{lem 1.9}(i)
     \[
     \Sigma_{A(\Gamma^2)}(\lambda)=\frac{(n_2+1)\lambda+2n_2\mu_2(v_1)}{\lambda^2-n_2}\\
     \]
     Two poles of $\Sigma_{A(\Gamma^2)}(\lambda)$ are $\lambda=\pm \sqrt{n_2}$. By Theorem \ref{thm 4.13}, the spectrum of $\Gamma^1 \boxminus \Gamma^2$ is given by
     \begin{itemize}
         \item The eigenvalue $0$ repeated $m_1-n_1+m_1(n_2-1)(=m_1n_2-n_1)$ times.
         \item Solving the roots of the equation
         \begin{equation*}
             \begin{split}
                 &\lambda^2-(1+\lambda \Sigma_{A(\Gamma^2)}(\lambda))(\gamma_1+\lambda_j(\Gamma^1))=0\\
                 i.e~~&\lambda^4-\big(n+(\gamma_1+\alpha_j)(n_2+2)\big)\lambda^2-2n_2(\gamma_1+\alpha_j)\mu_2(v_1)\lambda+n_2(\gamma_1+\alpha_j)=0\\
                 &\text{ corresponding to each eigenvalue } \alpha_j(j=1,2,\cdots,n_1) \text{ of } of \Gamma^1.
             \end{split}
         \end{equation*}
     \end{itemize}
     The remaining $2(m_1-n_1)$ eigenvalues must equal the two poles $\lambda=\pm \sqrt{n_2}$ of $\Sigma_{A(\Gamma^2)}(\lambda)$. By symmetry, we have $\sqrt{n_2}$ and $-\sqrt{n_2}$ as eigenvalues each with multiplicity $m_1-n_1$.
     \end{proof}
     \begin{cor} Consider $\gamma_1$-regular signed graph $\Gamma_1 = (G, \sigma_1, \mu_1)$ on $n_1$ vertices and $m_1$ edges. Let $\theta$ be any $r$-orientation of edges of $\Gamma^1$. Suppose $\Gamma^2$ be $(\gamma, k)$ co-regular signed graph on $n_2$ vertices and spectrum $(\beta_1, \beta_2, \cdots , \beta_{n_2} )$ where multiplicity of eigenvalue $k$ is $q$. If the spectrum of $\Gamma^1$ is $(\alpha_1, \alpha_2,\cdots , \alpha_{n_1} )$ then the spectrum of $(\Gamma^1 \boxminus \Gamma^2)_\theta$ consists of\\
     (i) Eigenvalue $0$ with multiplicity $m_1-n_1$.\\
     (ii) $m_1(n_2-q)$ eigenvalues $\beta_j$ each appearing with multiplicity $m_1$ corresponding to every eigenvalue $\beta_j(\neq k)$ of $\Gamma^2$.\\
     (iii) The roots of the equation
     $x^3-kx^2-(\gamma_1+\alpha_j)(1+n_2)x+k(\gamma_1+\alpha_j)=0$ corresponding to each eigenvalue $\alpha_j$ of $\Gamma^1$.\\
     (iv) Eigenvalue $k$ with multiplicity $m_1q-n_1$.
     \end{cor}
\begin{cor} Let $\Gamma^1$ be $\gamma_1$-regular signed graph and $\theta_1$ and $\theta_2$ be any $r$-orientations of edges of $\Gamma^1$. If $\Gamma^2$ and $\Gamma^3$ are two A-cospectral signed graphs such that $\Sigma_{A(\Gamma^2)}(\lambda)=\Sigma_{A(\Gamma^3)}(\lambda)$ then $(\Gamma^1\boxminus\Gamma^2)_{\theta_1}$ and $(\Gamma^1\boxminus\Gamma^3)_{\theta_2}$ are A-cospectral.
\end{cor}
\begin{thm}\label{thm 4.17} Consider $\gamma_1$-regular signed graph $\Gamma^1=(G,\sigma_1,\mu_1)$ on $n_1$ vertices and $m_1$ edges. Let  $\theta$ be any $r$-orientation of edges of $\Gamma^1$ and $\Gamma^2=(H,\sigma_2,\mu_2)$ be any arbitrary signed graph on $n_2$ vertices. If $\lambda$ is not a pole of $\Sigma_{L(\Gamma^2)}(\lambda-\gamma_1)$ then the characteristic polynomial of signed Laplacian matrix of $\Gamma^1 \boxminus \Gamma^2$ under the $r$-orientation $\theta$ is
\begin{equation*}
    \begin{split}
        f_{L(\Gamma^1 \boxminus \Gamma^2)_\theta}(\lambda)=(\lambda-2)^{m_1-n_1}\Big(f_{L(\Gamma^2)}(\lambda-2)\Big)^{m_1}~\prod_{j=1}^{n_1}\bigg(\lambda^2-(\gamma_1+2+2\gamma_1n_2)\lambda+2\gamma_1(n_2+1)\\
        -\Big(1+(\lambda-2)\Sigma_{L(\Gamma^2)}(\lambda-2)\Big)\Big(2\gamma_1-\lambda_j(L(\Gamma^1))\Big)\bigg)
    \end{split}
\end{equation*}
\end{thm}
\begin{proof} If we consider $R(\Gamma^1_\theta)$ as the vertex-edge incidence matrix of $\Gamma^1_\theta$, then with respect to the partition \ref{eqn 1.5} the signed Laplacian matrix of $(\Gamma^1 \boxminus \Gamma^2)_\theta$ is\\
\begin{equation*}
    L(\Gamma^1 \boxminus \Gamma^2)_\theta=\begin{bmatrix}
        \gamma_1(1+n_2)I_{n_1}&& -R(\Gamma^1_\theta)&& -R(\Gamma^1_\theta) \otimes \mu(\Gamma^2)^T\\ \\
        -R(\Gamma^1_\theta)^T&& 2I_{m_1}&& 0_{m_1 \times m_1} \otimes \mu(\Gamma^2)^T\\ \\
        -R(\Gamma^1_\theta)^T \otimes \mu(\Gamma^2)&& 0_{m_1 \times m_1} \otimes \mu(\Gamma^2)&& I_{m_1} \otimes \big(2I_{n_2}+L(\Gamma^2)\big)
    \end{bmatrix}
\end{equation*}
Thus
\begin{equation*}
    \begin{split}
        f_{L(\Gamma^1 \boxminus \Gamma^2)_\theta}(\lambda)&=det\begin{bmatrix}
        (\lambda-\gamma_1-\gamma_1n_2)I_{n_1}&& R(\Gamma^1_\theta)&& R(\Gamma^1_\theta) \otimes \mu(\Gamma^2)^T\\ \\
        R(\Gamma^1_\theta)^T&&(\lambda- 2)I_{m_1}&& 0_{m_1 \times m_1} \otimes \mu(\Gamma^2)^T\\ \\
        R(\Gamma^1_\theta)^T \otimes \mu(\Gamma^2)&& 0_{m_1 \times m_1} \otimes \mu(\Gamma^2)&& I_{m_1} \otimes \big((\lambda-2)I_{n_2}-L(\Gamma^2)\big)
    \end{bmatrix}\\ \\
    &=det\Big(I_{m_1} \otimes \big((\lambda-2)I_{n_2}-L(\Gamma^2)\big)\Big).~det(S)\\
    &=\big(f_{L(\Gamma^2)}(\lambda-2)\big)^{m_1}.~det(S)
    \end{split}
\end{equation*}
where $S=\begin{bmatrix}
    (\lambda-\gamma_1-\gamma_1n_2)I_{n_1}-\Sigma_{L(\Gamma^2)}(\lambda-2)R(\Gamma^1_\theta)R(\Gamma^1_\theta)^T&& R(\Gamma^1_\theta)\\ \\
    R(\Gamma^1_\theta)^T&& (\lambda-2)I_{m_1}
\end{bmatrix}$ is the schur complement \ref{schur} of $I_{m_1}\otimes \big((\lambda-2)I_{n_2}-L(\Gamma^2)\big)$.
\begin{equation*}
    \begin{split}
        f_{L(\Gamma^1 \boxminus \Gamma^2)_\theta}(\lambda)&=\big(f_{L(\Gamma^2)}(\lambda-2)\big)^{m_1}.~det\big[(\lambda-2)I_{m_1}\big].~det\Big[(\lambda-\gamma_1-\gamma_1n_2)I_{n_1}\\
        &~~~~~~~~~~~~~~~~~~~~~~~~~~~~~~~~~-\Sigma_{L(\Gamma^2)}(\lambda-2)R(\Gamma^1_\theta)R(\Gamma^1_\theta)^T-\frac{1}{\lambda-2}R(\Gamma^1_\theta)R(\Gamma^1_\theta)^T\Big]\\
        &=(\lambda-2)^{m_1}\big(f_{L(\Gamma^2)}(\lambda-2)\big)^{m_1}.~det\Big[(\lambda-\gamma_1-\gamma_1n_2)I_{n_1}\\
        &~~~~~~~~~~~~~~~~~~~~~~~~~~~~~~~~~~~~~~~~-\Big(\frac{1}{\lambda-2}+\Sigma_{L(\Gamma^2)}(\lambda-2)\Big)Q(\Gamma^1)\Big]
    \end{split}
\end{equation*}
As $\Gamma^1$ is $\gamma_1$-regular, $\lambda_j(Q(\Gamma^1))=\gamma_1+\lambda_j(\Gamma^1)=2\gamma_1-\lambda_j(L(\Gamma^1))$ for $j=1,2,\cdots,n_1$. Also using similar argument that we used in Theorem \ref{thm 4.17} we can say that the characteristic polynomial of $L(\Gamma^1 \boxminus \Gamma^2)_\theta$ is independent of $\theta$ and is given by
\begin{equation*}
    \begin{split}
        f_{L(\Gamma^1 \boxminus \Gamma^2)}(\lambda)&=(\lambda-2)^{m_1}\big(f_{L(\Gamma^2)}(\lambda-2)\big)^{m_1}~\prod_{i=1}^{n_1}\Big((\lambda-\gamma_1-\gamma_1n_2)\\
        &~~~~~~~~~~~~~~~~~~~~~~~~~~~~~~~~~~~~~~~~~~~~~~~~~~~-\big(\frac{1}{\lambda-2}+\Sigma_{L(\Gamma^2)}(\lambda-2)\big)\big(2\gamma_1-\lambda_j(L(\Gamma^1))\big)\Big)\\
        &=(\lambda-2)^{m_1-n_1}\big(f_{L(\Gamma^2)}(\lambda-\gamma_1)\big)^{m_1}~\prod_{j=1}^{n_1}\Big((\lambda^2-(\gamma_1+2+\gamma_1n_2)\lambda+2\gamma_1(n_2+1)\\
        &~~~~~~~~~~~~~~~~~~~~~~~~~~~~~~~~~~~~-\big(1+(\lambda-2)\Sigma_{L(\Gamma^2)}(\lambda-2)\big)\big(2\gamma_1-\lambda_j(L(\Gamma^1))\big)\Big) \qedhere
    \end{split}
\end{equation*}
\end{proof}
\begin{prop} Consider $\gamma_1$-regular signed graph $\Gamma^1=(G,\sigma_1,\mu_1)$ on $n_1$ vertices, $m_1$ edges. Let $\theta$ be any $r$-orientation of edges of $\Gamma^1$ and $\Gamma^2=(K_{1,n_2},\sigma_2,\mu_2)$ be a signed star with $V(\Gamma^2)=\{v_1,v_1,\cdots,v_{n_2+1}\}$ where $d(v_1)=n_2$ and $\mu_2=\mu^p$ or $\mu^c$. Suppose the spectrum of $L(\Gamma^1)$ is $(\alpha_1,\alpha_2,\cdots,\alpha_{n_{1}})$, then the spectrum of $L(\Gamma^1 \boxminus \Gamma_2)_\theta$ consists of\\
(i) Eigenvalue $2$ with multiplicity $2m_1-n_1$.\\
(ii) Eigenvalue $3$ with multiplicity $m_1(n_2-1)$.\\
(iii) Eigenvalue $n_2+3$ with multiplicity $m_1-n_1$.\\
(iv) The roots of the equation $x^3-(2\gamma_1+\gamma_1n_2+n_2+5)x^2+\big[(n_2+3)(\gamma_1n_2+2)+\alpha_j(n_2+2)+2\gamma_1\big]x+\big[5+3n_2+(n_2+\mu_2(v_1))^2\big](2\gamma_1-\alpha_j)=0$ corresponding to each eigenvalue $\alpha_j$ of $L(\Gamma^1)$.
\end{prop}
\begin{proof} The spectrum of $L(\Gamma^2)$ is $(0,1^{(n_2-1)},n_2+1)$. By Lemma \ref{lem 1.9}(ii) we have\\
\[
\Sigma_{L(\Gamma^2)}(\lambda-2)=\frac{(n_2+1)(\lambda-2)-(n_2^2+1)-2n_2\mu_2(v_1)}{(\lambda-2)(\lambda-2-n_2-1)}
\]
Two poles of $\Sigma_{L(\Gamma^2)}(\lambda-2)$ are $\lambda=2$ and $\lambda=n_2+3$. By Theorem \ref{thm 4.17}, the spectrum of $L(\Gamma^1\boxminus \Gamma^2)$ is given by
\begin{itemize}
    \item Eigenvalue $2$ with multiplicity $m_1-n_1$.
    \item Eigenvalue $3$ with multiplicity $m_1(n-1)$.
    \item The roots of the equation
    \begin{equation*}
             \begin{split}
             &(\lambda-2\gamma_1-\gamma_1n_2)(\lambda-2)-\Big(1+(\lambda-2)\Sigma_{L(\Gamma^2)}(\lambda-2)\Big)(2\gamma_1-\alpha_j)=0\\
             i.e~~&\lambda=2 \text{ with multiplicity $n_1$ and roots of the equation }\\
             &\lambda^3-(5+n_2+\gamma_1(n_2+2))\lambda^2+\big[(3+n_2)(\gamma_1n_2+2)+\alpha_j(n_2+2)+2\gamma_1\big]\lambda\\
             &~~~~~~~~~~~~~~~~~~~~~~~~~~~~~~~~~~~~~~~~~~~~+\big[5+3n_2+(n_2+\mu_2(v_1))^2\big](2\gamma_1-\alpha_j)=0\\
             &\text{ corresponding to each eigenvalue } \alpha_j(j=1,2,\cdots,n_1) \text{ of } L(\Gamma^1)
             \end{split}
         \end{equation*}
\end{itemize}
The remaining $2(m_1-n_1)$ eigenvalues must come from the poles $\lambda=2$ and $\lambda=n_2+3$ of $\Sigma_{L(\Gamma^2)}(\lambda-2)$. By symmetry, we have $2$ and $n_2+3$ as eigenvalues each with multiplicity $(m_1-n_1)$.
\end{proof}
\begin{cor} Consider $\gamma_1$-regular signed graph $\Gamma_1 = (G, \sigma_1, \mu_1)$ on $n_1$ vertices and $m_1$ edges. Let $\theta$ be any $r$-orientation of edges of $\Gamma^1$. Suppose $\Gamma^2$ be $(\gamma, k)$ co-regular signed graph on $n_2$ vertices and Laplacian spectrum $(\beta_1, \beta_2, \cdots , \beta_{n_2} )$ where multiplicity of eigenvalue $\gamma-k$ is $q$. If the spectrum of $L(\Gamma^1)$ is $(\alpha_1, \alpha_2,\cdots , \alpha_{n_1} )$ then the spectrum of $L(\Gamma^1 \boxminus \Gamma^2)_\theta$ consists of\\
     (i) Eigenvalue $2$ with multiplicity $m_1-n_1$.\\
     (ii) Eigenvalue $\beta_j+2$ each appearing with multiplicity $m_1$ corresponding to every eigenvalue $\beta_j(\neq \gamma-k)$ of $L(\Gamma^2)$.\\
     (iii) The roots of the equation
     $x^3-(\gamma_1+\gamma-k+\gamma_1n_2+4)x^2+\big[(\gamma_1+2+\gamma_1n_2)(2+\gamma-k)+\alpha_j(1+n_2)\big]x-2\gamma_1n_2(\gamma-k)-\alpha_j(2+\gamma-k+2n_2)=0$ for each eigenvalue $\alpha_j(j=1,2,\cdots,n_1)$ of $L(\Gamma^1)$.\\
     (iv) Eigenvalue $2+\gamma-k$ with multiplicity $m_1q-n_1$.
     \end{cor}
\begin{cor} Let $\Gamma^1$ be $\gamma_1$-regular signed graph and $\theta_1$ and $\theta_2$ be any $r$-orientations of edges of $\Gamma^1$. If $\Gamma^2$ and $\Gamma^3$ are two L-cospectral signed graphs such that $\Sigma_{L(\Gamma^2)}(\lambda)=\Sigma_{L(\Gamma^3)}(\lambda)$ then $(\Gamma^1\boxminus\Gamma^2)_{\theta_1}$ and $(\Gamma^1\boxminus\Gamma^3)_{\theta_2}$ are L-cospectral.
\end{cor}
\begin{thm}\label{thm 4.21} Consider $\gamma_1$-regular signed graph $\Gamma^1=(G,\sigma_1,\mu_1)$ on $n_1$ vertices and $m_1$ edges. Let $\theta$ be any $r$-orientation of edges of $\Gamma^1$ and $\Gamma^2=(H,\sigma_2,\mu_2)$ be any arbitrary signed graph on $n_2$ vertices. If $\lambda$ is not a pole of $\Sigma_{Q(\Gamma^2)}(\lambda-2)$ then the characteristic polynomial of signless Laplacian matrix of $\Gamma^1 \boxminus \Gamma^2$ under the $r$-orientation $\theta$ is
\begin{equation*}
    \begin{split}
        f_{Q(\Gamma^1 \boxminus \Gamma^2)_\theta}(\lambda)=(\lambda-2)^{m_1-n_1}\Big(f_{Q(\Gamma^2)}(\lambda-2)\Big)^{m_1}~\prod_{j=1}^{n_1}\bigg((\lambda-\gamma_1-\gamma_1n_2)(\lambda-2)\\
        -\Big(1+(\lambda-2)\Sigma_{Q(\Gamma^2)}(\lambda-2)\Big)\lambda_j(Q(\Gamma^1))\bigg)
    \end{split}
\end{equation*}
\end{thm}
\begin{proof} If we consider $R(\Gamma^1_\theta)$ as the vertex-edge incidence matrix of $\Gamma^1_\theta$, then with respect to the partition \ref{eqn 1.5} the signless Laplacian matrix of $(\Gamma^1 \boxminus \Gamma^2)_\theta$ is given by
\begin{equation*}
    Q(\Gamma^1 \boxminus \Gamma^2)_\theta=\begin{bmatrix}
        \gamma_1(1+n_2)I_{n_1}&& R(\Gamma^1_\theta)&& R(\Gamma^1_\theta) \otimes \mu(\Gamma^2)^T\\ \\
        R(\Gamma^1_\theta)^T&& 2I_{m_1}&& 0_{m_1 \times m_1} \otimes \mu(\Gamma^2)^T\\ \\
        R(\Gamma^1_\theta)^T \otimes \mu(\Gamma^2)&& 0_{m_1 \times m_1} \otimes \mu(\Gamma^2)&& I_{m_1} \otimes \big(2I_{n_2}+Q(\Gamma^2)\big)
    \end{bmatrix}
\end{equation*}
The remainder of the proof follows from Theorem \ref{thm 4.17}.
\end{proof}
\begin{cor} Consider $\gamma_1$-regular signed graph $\Gamma^1=(G,\sigma_1,\mu_1)$ on $n_1$ vertices, $m_1$ edges. Let $\theta$ be any $r$-orientation of edges of $\Gamma^1$ and $\Gamma^2=(K_{1,n_2},\sigma_2,\mu_2)$ be a signed star with $V(\Gamma^2)=\{v_1,v_1,\cdots,v_{n_2+1}\}$ where $d(v_1)=n_2$ and $\mu_2=\mu^p$ or $\mu^c$. Suppose the spectrum of $Q(\Gamma^1)$ is $(\alpha_1,\alpha_2,\cdots,\alpha_{n_{1}})$, then the spectrum of $Q(\Gamma^1 \boxminus \Gamma_2)_\theta$ consists of\\
(i) Eigenvalue $2$ with multiplicity $2m_1-n_1$.\\
(ii) Eigenvalue $3$ with multiplicity $m_1(n_2-1)$.\\
(iii) Eigenvalue $n_2+3$ with multiplicity $m_1-n_1$.\\
(iv) The roots of the equation $(x-2\gamma_1-\gamma_1n_2)(x-2)(x-3-n_2)+\alpha_j(n_2^2+3n_2+6-(n_2+2)x-2n_2\mu_2(v_1))=0$ for each eigenvalue $\alpha_j$ of $Q(\Gamma^1)$.
\end{cor}
\begin{cor} Consider $\gamma_1$-regular signed graph $\Gamma_1 = (G, \sigma_1, \mu_1)$ on $n_1$ vertices and $m_1$ edges. Let $\theta$ be any $r$-orientation of edges of $\Gamma^1$. Suppose $\Gamma^2$ be $(\gamma, k)$ co-regular signed graph on $n_2$ vertices and signless Laplacian spectrum $(\beta_1, \beta_2, \cdots , \beta_{n_2} )$ where multiplicity of eigenvalue $\gamma+k$ is $q$. If the spectrum of $Q(\Gamma^1)$ is $(\alpha_1, \alpha_2,\cdots , \alpha_{n_1} )$ then the spectrum of $Q(\Gamma^1 \boxminus \Gamma^2)_\theta$ consists of\\
     (i) Eigenvalue $2$ with multiplicity $m_1-n_1$.\\
     (ii) Eigenvalue $\beta_j+2$ each appearing with multiplicity $m_1$ corresponding to every eigenvalue $\beta_j(\neq \gamma+k)$ of $Q(\Gamma^2)$.\\
     (iii) The roots of the equation
     $(x-\gamma_1-\gamma_1n_2)(x-2)(x-2-\gamma-k)-\alpha_j(x-2-\gamma-k+n_2(x-2))=0$ corresponding to every eigenvalue $\alpha_j$ of $Q(\Gamma^1)$.\\
     (iv) Eigenvalue $2+\gamma+k$ with multiplicity $m_1q-n_1$.
     \end{cor}
     \begin{cor} Let $\Gamma^1$ be $\gamma_1$-regular signed graph and $\theta_1$ and $\theta_2$ be any $r$-orientations of edges of $\Gamma^1$. If $\Gamma^2$ and $\Gamma^3$ are two Q-cospectral signed graphs such that $\Sigma_{Q(\Gamma^2)}(\lambda)=\Sigma_{Q(\Gamma^3)}(\lambda)$ then $(\Gamma^1\boxminus\Gamma^2)_{\theta_1}$ and $(\Gamma^1\boxminus\Gamma^3)_{\theta_2}$ are Q-cospectral.
     \end{cor}
\section[Short Title for TOC and Headers]{Normalized Laplacian spectrum of $(\Gamma^1\diamond\Gamma^2)_\theta$, $(\Gamma^1\boxdot\Gamma^2)_\theta$ and $(\Gamma^1\boxminus\Gamma^2)_\theta$}\label{normalized}

\begin{lem}\label{lem 5.1} Consider any signed graph $\Gamma^1=(G_1,\sigma_1,\mu_1)$ on $n_1$ vertices, $m_1$ edges and $\theta$ be any $r$-orientation of edges of $\Gamma^1$. Let $\Gamma^2=(G_2,\sigma_2,\mu_2)$ be $\gamma$-regular signed graph on $n_2$ vertices then\\
\[
P(\Gamma^1 \diamond \Gamma^2)_\theta=\begin{bmatrix}
    \frac{1}{n_2+1}P(\Gamma^1)&& \frac{1}{n_2+1}D(\Gamma^1)^{-1}\big(R(\Gamma^1_\theta)\otimes \mu(\Gamma^2)^T\big)\\ \\
    \frac{1}{\gamma+2}\big(R(\Gamma^1_\theta)^T\otimes \mu(\Gamma^2)\big)&& \frac{\gamma}{\gamma+2}\big(I_{m_1}\otimes P(\Gamma^2)\big)
\end{bmatrix}
\]
\end{lem}
\begin{proof} From equation \ref{eqn 1.3} we have\\
\[
D(\Gamma^1 \diamond \Gamma^2)_\theta=\begin{bmatrix}
    (1+n_2)D(\Gamma^1) && 0\\ \\
    0 && (2+\gamma)I_{m_1n_2}
\end{bmatrix}
\]
Thus
\begin{equation*}
\begin{split}
    P(\Gamma^1 \diamond \Gamma^2)_\theta&=D(\Gamma^1 \diamond \Gamma^2)^{-1}_\theta A(\Gamma^1 \diamond \Gamma^2)_\theta\\ \\
    &=\begin{bmatrix}
        \frac{1}{1+n_2}D(\Gamma^1)^{-1} && 0\\
        0 && \frac{1}{2+\gamma} I_{m_1n_2}
    \end{bmatrix}
    \begin{bmatrix}
        A(\Gamma^1) && R(\Gamma^1_\theta) \otimes \mu(\Gamma^2)^T\\
        R(\Gamma^1_\theta)^T \otimes \mu(\Gamma^2) && I_{m_1} \otimes A(\Gamma^2)
    \end{bmatrix}\\ \\
    &=\begin{bmatrix}
        \frac{1}{n_2+1}P(\Gamma^1)&& \frac{1}{n_2+1}D(\Gamma^1)^{-1}\big(R(\Gamma^1_\theta)\otimes \mu(\Gamma^2)^T\big)\\ \\
    \frac{1}{\gamma+2}\big(R(\Gamma^1_\theta)^T\otimes \mu(\Gamma^2)\big)&& \frac{1}{\gamma+2}\big(I_{m_1}\otimes A(\Gamma^2)\big)
    \end{bmatrix}
    \end{split}
\end{equation*}
The result follow using the fact that $\Gamma^2$ is $\gamma$-regular.
\end{proof}
\begin{lem}\label{lem 5.2} Let $\Gamma^j=(G_j,\sigma_j,\mu_j)$ be $\gamma_j$-regular signed graph on $n_j$ vertices and $m_j$ edges, $j=1,2$ and $\theta$ be any $r$-orientation of edges of $\Gamma^1$ then,\\
\begin{equation*}
    \begin{split}
        P(\Gamma^1 \boxdot \Gamma^2)_\theta=\begin{bmatrix}
            0_{n_1} && \frac{1}{\gamma_1}R(\Gamma^1_\theta) && 0_{n_1\times n_1n_2}\\ \\
            \frac{1}{2+2n_2}R(\Gamma^1_\theta)^T && 0_{m_1} && \frac{1}{2+2n_2}\big(R(\Gamma^1_\theta)^T\otimes \mu(\Gamma^2)^T\big)\\ \\
            0_{n_1n_2\times n_1} && \frac{1}{\gamma_1+\gamma_2} \big(R(\Gamma^1_\theta)\otimes \mu(\Gamma^2)\big) && \frac{\gamma_2}{\gamma_1+\gamma_2}\big(I_{n_1}\otimes P(\Gamma^2)\big)
        \end{bmatrix}
    \end{split}
\end{equation*}
\end{lem}
\begin{proof} Proof is similar to that of Lemma \ref{lem 5.1}.
\end{proof}
\begin{lem}\label{lem 5.3}
 Consider a signed graph $\Gamma^1=(G_1,\sigma_1,\mu_1)$ on $n_1$ vertices and $m_1$ edges. Let $\theta$ be any $r$-orientation of edges of $\Gamma^1$ and $\Gamma^2=(G_2,\sigma_2,\mu_2)$ be $\gamma$-regular signed graph on $n_2$ vertices then\\
\[
P(\Gamma^1 \boxminus \Gamma^2)_\theta=\begin{bmatrix}
            0_{n_1} && \frac{1}{1+n_2}D(\Gamma^1)^{-1}R(\Gamma^1_\theta) && \frac{1}{1+n_2}D(\Gamma^1)^{-1}\big(R(\Gamma^1_\theta)\otimes \mu(\Gamma^2)^T\big)\\ \\
            \frac{1}{2}R(\Gamma^1_\theta)^T && 0_{m_1} && 0_{m_1 \times m_1n_2}\\ \\
            \frac{1}{2+\gamma} \big(R(\Gamma^1_\theta)^T\otimes \mu(\Gamma^2)\big) && 0_{m_1n_2 \times m_1} && \frac{\gamma}{\gamma+2}\big(I_{m_1}\otimes P(\Gamma^2)\big)
        \end{bmatrix}
\]
\end{lem}
\begin{proof} Proof is similar to Lemma \ref{lem 5.1}
\end{proof}
\begin{thm} Let $\Gamma^1=(G_1,\sigma_1,\mu_1)$ be any signed graph with $n_1$ vertices, $m_1$ edges and $\theta$ be any $r$-orientation of edges of $\Gamma^1$ and $\Gamma^2=(G_2,\sigma_2,\mu_2)$ be $\gamma$-regular signed graph on $n_2$ vertices then the characteristic polynomial of $P(\Gamma^1 \diamond \Gamma^2)_\theta$ is\\
\[
f_{P(\Gamma^1 \diamond \Gamma^2)_\theta}(\lambda)=\prod_{j=1}^{n_2}\big(\lambda-\frac{\gamma}{\gamma+2}\lambda_j(P(\Gamma^2))\big)^{m_1}~\prod_{j=1}^{n_1}\Big[\lambda-\frac{\lambda_j(P(\Gamma^1))}{n_2+1}-\frac{1+\lambda_j(P(\Gamma^1))}{n_2+1}\Sigma_{A(\Gamma^2)}(\gamma \lambda+2\lambda)\Big]
\]
\end{thm}
\begin{proof} Using Lemma \ref{lem 5.1} we can write
\begin{equation*}
    \begin{split}
        f_{P(\Gamma^1 \diamond \Gamma^2)_\theta}(\lambda)&=det \begin{bmatrix}
            \lambda I_{n_1}-\frac{1}{1+n_2}P(\Gamma^1) && -\frac{1}{1+n_2}D(\Gamma^1)^{-1}\left(R(\Gamma^1_\theta)\otimes \mu(\Gamma^2)^T\right)\\ \\
            \frac{1}{\gamma+2}\left(R(\Gamma^1_\theta)^T \otimes \mu(\Gamma^2)\right) && \lambda I_{m_1n_2}-\frac{\gamma}{\gamma+2}\left(I_{m_1}\otimes P(\Gamma^2)\right)
        \end{bmatrix}\\
        &=det\left[I_{m_1}\otimes \left(\lambda I_{n_2}-\frac{\gamma}{\gamma+2}P(\Gamma^2)\right)\right] det\Bigg[\lambda I_{n_1}-\frac{1}{n_2+1}P(\Gamma^1)\\ &~~~~~~-\frac{1}{(n_2+1)(\gamma+2)}D(\Gamma^1)^{-1}R(\Gamma_\theta^1)R(\Gamma_\theta^1)^T\otimes \Big(\mu(\Gamma^2)^T\big(\lambda I_{n_2}-\frac{\gamma}{\gamma+2}P(\Gamma^2)\big)^{-1}\mu(\Gamma^2)\Big)\Bigg]\\
        &=det\left[I_{m_1}\otimes \left(\lambda I_{n_2}-\frac{\gamma}{\gamma+2}P(\Gamma^2)\right)\right] det\Big[\lambda I_{n_1}-\frac{1}{n_2+1}P(\Gamma^1)\\&~~~~~~~~~~~~~~~~~~~~~~~~~~-\frac{1}{(n_2+1)(\gamma+2)}\left(I_{n_1}+P(\Gamma^1)\right)\otimes (\gamma+2)\Sigma_{A(\Gamma^2)}\left(\lambda \gamma+2\lambda\right)\Big]\\
        &=\prod_{j=1}^{n_2}\left(\lambda-\frac{\gamma}{\gamma+2}\lambda_j(P(\Gamma^2))\right)^{m_1} \prod_{j=1}^{n_1}\Big[\lambda-\frac{\lambda_j(P(\Gamma^1))}{n_2+1}-\frac{1+\lambda_j(P(\Gamma^1))}{n_2+1}\Sigma_{A(\Gamma^2)}\left(\lambda \gamma+2\lambda\right)\Big] \qedhere
    \end{split}
\end{equation*}
\end{proof}
\begin{cor} Let $\Gamma^1=(G_1,\sigma_1,\mu_1)$ be any signed graph with $n_1$ vertices, $m_1$ edges and $\theta$ be any $r$-orientation of edges of $\Gamma^1$ and $\Gamma^2=(G_2,\sigma_2,\mu_2)$ be $\gamma$-regular signed graph on $n_2$ vertices then the characteristic polynomial of $\mathbb{L}(\Gamma^1 \diamond \Gamma^2)_\theta$ is
\begin{align*}
&f_{\mathbb{L}(\Gamma^1 \diamond \Gamma^2)_\theta}(\lambda)\\=&\prod_{j=1}^{n_2}\left[\frac{2+\gamma\lambda_j(\mathbb{L}(\Gamma^2))}{\gamma+2}-\lambda\right]^{m_1}~\prod_{j=1}^{n_1}\left[\frac{n_2+\lambda_j(\mathbb{L}(\Gamma^1))+(2-\lambda_j(\mathbb{L}(\Gamma^1)))\Sigma_{A(\Gamma^2)}(\gamma+2-\gamma\lambda-2\lambda)}{n_2+1}-\lambda\right]
\end{align*}
\end{cor}
\begin{cor} Let $\Gamma^1$ be signed graph and $\theta_1$ and $\theta_2$ be any $r$-orientations of edges of $\Gamma^1$. If $\Gamma^2$ and $\Gamma^3$ are two $\gamma$-regular $\mathbb{L}$-cospectral signed graphs such that $\Sigma_{A(\Gamma^2)}(\lambda)=\Sigma_{A(\Gamma^3)}(\lambda)$ then $(\Gamma^1\diamond\Gamma^2)_{\theta_1}$ and $(\Gamma^1\diamond\Gamma^3)_{\theta_2}$ are $\mathbb{L}$-cospectral.
\end{cor}
\begin{thm}\label{thm 5.7} Let $\Gamma^j=(G_j,\sigma_j,\mu_j)$ be $\gamma_j$-regular signed graph on $n_j$ vertices and $m_j$ edges, $j=1,2$ and $\theta$ be any $r$-orientation of edges of $\Gamma^1$ then the characteristic polynomial of $P(\Gamma^1 \boxdot \Gamma^2)_\theta$ is
\begin{align*}
  &f_{P(\Gamma^1\boxdot\Gamma^2)_\theta}(\lambda)\\=&\lambda^{m_1-n_1}.\prod_{j=1}^{n_2}\left[\lambda-\frac{\gamma_2}{\gamma_1+\gamma_2}\lambda_j(P(\Gamma^2))\right].\prod_{j=1}^{n_1}\left[\lambda^2-\frac{1+\lambda_j(P(\Gamma^1))}{2(n_2+1)}\left(1+\lambda \gamma_1\Sigma_{A(\Gamma^2)}((\gamma_1+\gamma_2)\lambda)\right)\right] 
\end{align*}
\end{thm}
\begin{proof} Using Lemma \ref{lem 5.2} we can write\\
\begin{equation*}
    \begin{split}
        f_{P(\Gamma^1 \boxdot \Gamma^2)_\theta}(\lambda)&= det \begin{bmatrix}
            \lambda I_{n_1} && -\frac{1}{\gamma_1}R(\Gamma^1_\theta) && 0_{n_1\times n_1n_2}\\ \\
            -\frac{1}{2+2n_2}R(\Gamma^1_\theta)^T && \lambda I_{m_1} && -\frac{1}{2+2n_2}\big(R(\Gamma^1_\theta)^T\otimes \mu(\Gamma^2)^T\big)\\ \\
            0_{n_1n_2\times n_1} && -\frac{1}{\gamma_1+\gamma_2} \big(R(\Gamma^1_\theta)\otimes \mu(\Gamma^2)\big) && I_{n_1}\otimes \left(\lambda I_{n_2}-\frac{\gamma_2}{\gamma_1+\gamma_2} P(\Gamma^2)\right)
        \end{bmatrix}\\ \\
        &= det \left[I_{n_1}\otimes (\lambda I_{n_2}-\frac{\gamma_2}{\gamma_1+\gamma_2}P(\Gamma^2))\right].~det(S)
    \end{split}
\end{equation*}
where $S=\begin{bmatrix}
    \lambda I_{n_1} && -\frac{1}{\gamma_1}R(\Gamma^1_\theta)\\ \\
    -\frac{1}{2(n_2+1)}R(\Gamma^1_\theta)^T && \lambda I_{m_1}-\frac{1}{2(n_2+1)}\Sigma_{A(\Gamma^2)}((\gamma_1+\gamma_2)\lambda)R(\Gamma^1_\theta)^TR(\Gamma^1_\theta)
\end{bmatrix}$ is a Schur complement \ref{schur} of $I_{n_1}\otimes (\lambda I_{n_2}-\frac{\gamma_2}{\gamma_1+\gamma_2}P(\Gamma^2))$. Thus,
\begin{equation*}
    \begin{split}
        &f_{P(\Gamma^1\boxdot\Gamma^2)_\theta}(\lambda)\\=&\lambda^{n_1}.\prod_{j=1}^{n_2}\left[\lambda-\frac{\gamma_2}{\gamma_1+\gamma_2}\lambda_j(P(\Gamma^2))\right].det\left[\lambda I_{m_1}-\frac{1}{2(n_2+1)}\left[\frac{1}{\lambda \gamma_1}+\Sigma_{A(\Gamma^2)}((\gamma_1+\gamma_2)\lambda)\right]R(\Gamma^1_\theta)^TR(\Gamma^1_\theta)\right]
    \end{split}
\end{equation*}
As $\Gamma^1$ is $\gamma_1$-regular, $A(\Gamma^1)=\gamma_1P(\Gamma^1)$. Applying Remark \ref{rem 1.7} and Lemma \ref{lem 1.8} as in Theorem \ref{thm 4.1} we get
\begin{equation*}
  f_{P(\Gamma^1\boxdot\Gamma^2)_\theta}(\lambda)=\lambda^{m_1-n_1}.\prod_{j=1}^{n_2}\left[\lambda-\frac{\gamma_2}{\gamma_1+\gamma_2}\lambda_j(P(\Gamma^2))\right].\prod_{j=1}^{n_1}\left[\lambda^2-\frac{1+\lambda_j(P(\Gamma^1))}{2(n_2+1)}\left(1+\lambda \gamma_1\Sigma_{A(\Gamma^2)}(\gamma_1+\gamma_2)\lambda)\right)\right] 
\end{equation*}
\end{proof}
\begin{cor} Let $\Gamma^j=(G_j,\sigma_j,\mu_j)$ be $r_j$-regular signed graph on $n_j$ vertices and $m_j$ edges, $j=1,2$ and $\theta$ be any $r$-orientation of edges of $\Gamma^1$ then the characteristic polynomial of $\mathbb{L}(\Gamma^1 \boxdot \Gamma^2)_\theta$ is
\begin{equation*}
\begin{split}
  f_{\mathbb{L}(\Gamma^1\boxdot\Gamma^2)_\theta}(\lambda)=&(1-\lambda)^{m_1-n_1}.\prod_{j=1}^{n_2}\left[\frac{\gamma_1+\gamma_2\lambda_j(\mathbb{L}(\Gamma^2))}{\gamma_1+\gamma_2}-\lambda\right].\prod_{j=1}^{n_1} \bigg[(1-\lambda)^2-\\ &~~~~~~~~~~~~~~~~~~~~~~~~~~~~\frac{2-\lambda_j(\mathbb{L}(\Gamma^1))}{2(n_2+1)}\left(1+(1-\lambda) \gamma_1\Sigma_{A(\Gamma^2)}(\gamma_1+\gamma_2-\gamma_1\lambda-\gamma_2\lambda)\right)\bigg]
  \end{split}
\end{equation*}
\end{cor}
\begin{cor} Let $\Gamma^1$ be $\gamma_1$-regular signed graph and $\theta_1$ and $\theta_2$ be any $r$-orientations of edges of $\Gamma^1$. If $\Gamma^2$ and $\Gamma^3$ are two $\gamma_2$-regular $\mathbb{L}$-cospectral signed graphs such that $\Sigma_{A(\Gamma^2)}(\lambda)=\Sigma_{A(\Gamma^3)}(\lambda)$ then $(\Gamma^1\boxdot\Gamma^2)_{\theta_1}$ and $(\Gamma^1\boxdot\Gamma^3)_{\theta_2}$ are $\mathbb{L}$-cospectral.
\end{cor}
\begin{thm} Consider a signed graph $\Gamma^1=(G_1,\sigma_1,\mu_1)$ on $n_1$ vertices and $m_1$ edges. Let $\theta$ be any $r$-orientation of edges of $\Gamma^1$ and $\Gamma^2=(G_2,\sigma_2,\mu_2)$ be $\gamma$-regular signed graph on $n_2$ vertices then the characteristic polynomial of $P(\Gamma^1 \boxminus \Gamma^2)_\theta$ is
\begin{align*}
  &f_{P(\Gamma^1\boxminus\Gamma^2)_\theta}(\lambda)\\=&\lambda^{m_1-n_1}.\prod_{j=1}^{n_2}\left[\lambda-\frac{\gamma}{\gamma+2}\lambda_j(P(\Gamma^2))\right].\prod_{j=1}^{n_1}\left[\lambda^2-\frac{1}{2(n_2+1)}\left(1+2\Sigma_{A(\Gamma^2)}(\gamma\lambda+2\lambda)\right)(1+\lambda_j(P(\Gamma^1)))\right] 
\end{align*}
\end{thm}
\begin{proof}
    Proof is similar to Theorem \ref{thm 5.7}
\end{proof}
\begin{cor} Consider a signed graph $\Gamma^1=(G_1,\sigma_1,\mu_1)$ on $n_1$ vertices and $m_1$ edges. Let $\theta$ be any $r$-orientation of edges of $\Gamma^1$ and $\Gamma^2=(G_2,\sigma_2,\mu_2)$ be $\gamma$-regular signed graph on $n_2$ vertices then the characteristic polynomial of $\mathbb{L}(\Gamma^1 \boxminus \Gamma^2)_\theta$ is
\begin{equation*}
\begin{split}
  f_{\mathbb{L}(\Gamma^1\boxminus\Gamma^2)_\theta}(\lambda)=&(1-\lambda)^{m_1-n_1}.\prod_{j=1}^{n_2}\left[\frac{2+\gamma\lambda_j(\mathbb{L}(\Gamma^2))}{\gamma+2}-\lambda\right].\prod_{j=1}^{n_1}\bigg[(1-\lambda)^2-\frac{1}{2(n_2+1)}\\&~~~~~~~~~~~~~~~~~~~~~~~~~~~~~~~~~~~~~~~~\left(1+2\Sigma_{A(\Gamma^2)}(\gamma+2-\gamma\lambda-2\lambda)\right)(2-\lambda_j(\mathbb{L}(\Gamma^1)))\bigg] 
  \end{split}
  \end{equation*}
  \end{cor}
  \begin{cor} Let $\Gamma^1$ be any signed graph and $\theta_1$ and $\theta_2$ be any $r$-orientations of edges of $\Gamma^1$. If $\Gamma^2$ and $\Gamma^3$ are two $\gamma$-regular  $\mathbb{L}$-cospectral signed graphs such that $\Sigma_{A(\Gamma^2)}(\lambda)=\Sigma_{A(\Gamma^3)}(\lambda)$ then $(\Gamma^1\boxminus\Gamma^2)_{\theta_1}$ and $(\Gamma^1\boxminus\Gamma^3)_{\theta_2}$ are $\mathbb{L}$-cospectral.
  \end{cor}
  \section*{Disclosure statement} No potential conflict of interest.
  \section*{Acknowledgements}
We would like to acknowledge National Institute of Technology Sikkim for giving doctoral fellowship to Satyam Guragain and Bishal Sonar.
  \bibliographystyle{plain}
\bibliography{main.bib}
\end{document}